\documentclass[12pt]{amsart}
\usepackage{amssymb}
\usepackage{mathtools}
\usepackage[english]{babel}
\usepackage{epsfig}
\usepackage{mathptmx}
\usepackage{amsmath,amsfonts,amssymb}
\usepackage{mathtools}
\usepackage{mathrsfs}
\usepackage[all]{xy}
\usepackage{graphicx}
\usepackage{latexsym}
\usepackage{verbatim}
\usepackage{enumerate}
\usepackage[svgnames]{xcolor}
\usepackage{ulem}
\usepackage{cancel}
\usepackage{tabularx}
\usepackage{caption}
\usepackage{glossaries}
\setglossarystyle{long4colheader}

\usepackage{wasysym}

\newcommand{\hide}[1]{}
\usepackage{pifont}

\numberwithin{equation}{section}

\usepackage[backref=page]{hyperref} 

\def\Re{{\sf Re}\,}

\newcommand{\D}{\mathbb D}

\newcommand{\R}{\mathbb R}

\newcommand{\Z}{\mathbb Z}
\newcommand{\C}{\mathbb C}

\newcommand{\Aut}{{\sf Aut}(\mathbb D)}

\newcommand{\N}{\mathbb N}

\def\Re{{\sf Re}\,}

\newcommand{\Jacobi}{\mathtt{P}}

\newcommand{\F}{{}_2F_1}
\newcommand{\Fz}[3]{{}_2F_1\left(#1;#2;#3\right)}

\def\Aut{{\sf Aut}}

\def\Re{{\sf Re}\,}

\def\Re{{\sf Re}\,}

\def\1#1{\overline{#1}}
\def\2#1{\widetilde{#1}}
\def\3#1{\widehat{#1}}
\def\4#1{\mathbb{#1}}
\def\5#1{\frak{#1}}
\def\6#1{{\mathcal{#1}}}

\def\Re{{\sf Re}\,}

\newcommand{\mcite}[1]{\csname b@#1\endcsname}

\newtheoremstyle{break}
{8pt}{8pt}%
{\itshape}{}%
{\bfseries}{}%
{\newline}{}%
\theoremstyle{break}

\theoremstyle{theorem}

\def\Aut{{\sf Aut}}

\def\Re{{\sf Re}\,}

\newcommand{\cc}[1]{\overline{{#1}}}

\newcommand{\del}{\mathop{}\!\partial}

\newcommand{\rotdiagonal}{\mathrm{L}_{\widehat{\C}}}
\newcommand{\diagonal}{\mathrm{L}_{\D}}

\allowdisplaybreaks

\mathtoolsset{showonlyrefs,showmanualtags}

\theoremstyle{break}

\newtheorem{theorem}{Theorem}[section]
\newtheorem*{theorem*}{Theorem}
\newtheorem{lemma}[theorem]{Lemma}
\newtheorem*{lemma*}{Lemma}
\newtheorem{proposition}[theorem]{Proposition}
\newtheorem{corollary}[theorem]{Corollary}

\theoremstyle{break}
\newtheorem{definition}[theorem]{Definition}

\theoremstyle{remark}
\newtheorem{remark}{Remark}

\numberwithin{equation}{section}

\newmuskip\pFqmuskip

\newcommand*\pFq[6][8]{%
	\begingroup 
	\pFqmuskip=#1mu\relax
	\mathchardef\normalcomma=\mathcode`,
	\mathcode`\,=\string"8000
	\begingroup\lccode`\~=`\,
	\lowercase{\endgroup\let~}\pFqcomma
	{}_{#2}F_{#3}{\left[\genfrac..{0pt}{}{#4}{#5};#6\right]}%
	\endgroup
}
\newcommand{\pFqcomma}{{\normalcomma}\mskip\pFqmuskip}

\newcommand*\FF[3][8]{%
	\begingroup 
	\pFqmuskip=#1mu\relax
	\mathchardef\normalcomma=\mathcode`,
	\mathcode`\,=\string"8000
	\begingroup\lccode`\~=`\,
	\lowercase{\endgroup\let~}\pFqcomma
	{}_{3}F_{2}{\left[\genfrac..{0pt}{}{#2}{#3}\right]}%
	\endgroup
}

\newcommand{\FFz}[3]{\pFq{3}{2}{#1}{#2}{#3}}

\makeatletter
\def\blfootnote{\xdef\@thefnmark{}\@footnotetext}
\makeatother

\author[A.~Moucha]{Annika Moucha$^\dag$}
\address{A. Moucha: Department of Mathematics, University of W\"urzburg, Emil Fischer Strasse 40, 97074, W\"urzburg, Germany.} \email{annika.moucha@uni-wuerzburg.de}

\title{Spectral synthesis of the invariant Laplacian and complexified spherical harmonics}
\thanks{$^\dag\,$Partially supported by the Alexander von Humboldt Stiftung}

\begin{document}

	\maketitle



	\begin{abstract}
		We show that the space $\mathcal{H}(\Omega)$ of holomorphic functions $F:\Omega\to\C$, where ${\Omega=\{(z,w)\in\widehat{\C}^2\,:\, z\cdot w\neq 1\}}$, possesses an orthogonal Schauder basis consisting of distinguished eigenfunctions of the	canonical Laplacian on $\Omega$. Mapping $\Omega$ biholomorphically onto the complex two-sphere, we use the Schauder basis result in order to identify the classical three-dimensional spherical harmonics as restrictions of the elements in $\mathcal{H}(\Omega)$ to the real two-sphere analogue in $\Omega$. In particular, we show that the zonal harmonics correspond to those functions in $\mathcal{H}(\Omega)$ that are invariant under automorphisms of $\Omega$ induced by M\"obius transformations. The proof of the Schauder basis result is based on a curious combinatorial identity which we prove with the help of generalized hypergeometric functions.
	\end{abstract}

	\section{Introduction}\label{sec:intro}\

	Let $\widehat{\C}:=\C\cup\{\infty\}$ be the \emph{Riemann sphere} and define\footnote{We use the convention $\infty\cdot0=0\cdot\infty=1$ and $\infty\cdot z=z\cdot\infty=\infty$ for $z\in\widehat{\C}\setminus\{0\}$.}
	\begin{equation}\label{eq:Omega}
		\Omega:=\{(z,w)\in\widehat{\C}^2\,:\, z\cdot w\neq 1\}\,.
	\end{equation}
	The set $\Omega$ encompasses $\widehat{\C}$ and the \emph{open unit disk} $\D$ in a natural way: there is a one-to-one correspondence between $\widehat{\C}$ and
	the ``rotated diagonal'' in $\Omega$, that is
	\begin{equation}\label{eq:rotdiagonal}
		\rotdiagonal:=\{(z,-\cc{z})\,:\, z\in\widehat{\C}\}\,,
	\end{equation}
	as well as between $\D$	and the ``(half)  diagonal'' in $\Omega$, that is
	\begin{equation}\label{eq:diagonal}
		\diagonal:=\{(z,\cc{z})\,:\, z\in\D\}\,.
	\end{equation}
	\begin{figure}[t]
		\centering
		\includegraphics[width = 0.9\textwidth]{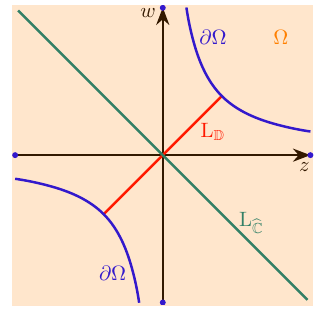}
		\caption{Schematic picture of the domain $\Omega$}
		\label{fig:Omega}
	\end{figure}Consequently, every function $F:\Omega\to\C$ gives rise to its ``shadows'' on $\widehat{\C}$ and $\D$, i.e.\ to the functions
	\[F\big\vert_{\rotdiagonal}:\widehat{\C}\rightarrow\C\quad\text{and}\quad F\big\vert_{\diagonal}:\D\rightarrow\C\,.\]
	If we additionally assume that $F$ is a holomorphic function, i.e.\ $F$ belongs to the set
	\begin{equation}
		\mathcal{H}(\Omega):=\{F:\Omega\rightarrow\C\,:\, F\text{ is holomorphic}\}\,,
	\end{equation}
	then its ``shadows'' ${F\vert}_{\rotdiagonal}$ resp.\ ${F\vert}_{\diagonal}$ belong to $C^\infty(\widehat{\C})$ resp.\ $C^\infty(\D)$, by which we denote the spaces of smooth, that is infinitely real differentiable, functions on $\widehat{\C}$ resp.\ on $\D$.  In fact, this way we even establish a one-to-one correspondence between $\mathcal{H}(\Omega)$ and its ``shadow spaces''
	\begin{align}\label{eq:starproductC}
		\mathcal{A}(\widehat{\C})&:=\{f\in C^\infty(\widehat{\C})\,:\, f={F\vert}_{\rotdiagonal} \text{ for some }F\in\mathcal{H}(\Omega)\}\quad\text{and}\\
		\mathcal{A}(\D)&:=\{f\in C^\infty(\D)\,:\, f={F\vert}_{\diagonal} \text{ for some }F\in\mathcal{H}(\Omega)\}\,.\label{eq:starproductD}
	\end{align}
	Indeed, by a variant of the identity principle (see \cite[p.~18]{Range}), if $U \subseteq \C^2$ is a domain containing a point of the form $(z,-\cc{z})$ or $(z,\cc{z})$ and ${F,G: U \to \C}$ are holomorphic and agree on $\rotdiagonal\cap U$ or $\diagonal\cap U$, then $F\equiv G$ on $U$. Further, the identification of elements in $\mathcal{H}(\Omega)$ with its ``shadows'' allows us to equip $\mathcal{A}(\widehat{\C})$ and $\mathcal{A}(\D)$ with the topology induced by the natural topology of locally uniform convergence on $\mathcal{H}(\Omega)$. This way, $\mathcal{A}(\widehat{\C})$ and $\mathcal{A}(\D)$ become Fr\'echet spaces.

	\medskip

	The ``shadow spaces'' $\mathcal{A}(\widehat{\C})$ and $\mathcal{A}(\D)$ appear in recent work on strict deformation quantization of $\widehat{\C}$ resp.\ $\D$ by Esposito, Schmitt, Waldmann \cite{EspositoSchmittWaldmann2019} resp.\ Kraus, Roth, Sch\"otz, Waldmann \cite{KrausRothSchoetzWaldmann2019}, see also \cite{BeiserWaldmann2014,KrausRothSchleissinger,KrausRothSchoetzWaldmann2019,SchmittSchoetz2022}. Roughly speaking, the goal in \cite{EspositoSchmittWaldmann2019} resp.\ \cite{KrausRothSchoetzWaldmann2019} is the construction of a so-called star product, a distinguished non-commutative product, on $\mathcal{A}(\D)$ resp.\ $\mathcal{A}(\widehat{\C})$. It turns out that, in this context, a pivotal tool are basic invariant differential operators on $C^\infty(\widehat{\C})$ and $C^\infty(\D)$, the so-called Peschl-Minda operators, which were already studied by Peschl \cite{Peschl1955}, Minda \cite{Minda}, Kim and Sugawa \cite{KS07diff} and others, see e.g.\ \cite{Ahar69,Harmelin1982,Schippers2003,Schippers2007,KS09}. This was recently realized by Heins, Roth, Sugawa and the author \cite{HeinsMouchaRoth2}. In particular, they show that these Peschl-Minda operators are, also, ``shadows'' of invariant differential operators on $\mathcal{H}(\Omega)$. This gives a partial explanation for the appearance of $\mathcal{H}(\Omega)$ in strict deformation quantization of $\widehat{\C}$ resp.~$\D$: the Peschl-Minda operators (on $\mathcal{H}(\Omega)$) and the corresponding star product on $\mathcal{H}(\Omega)$ (see \cite[Th.~6.6]{HeinsMouchaRoth2}) provide a unifying framework for the (real) theories of $\mathcal{A}(\widehat{\C})$ and $\mathcal{A}(\D)$.

	\medskip

	Another fruitful connection between $\mathcal{H}(\Omega)$ with its ``shadow spaces'' and smooth functions is given by the spectral theory of the \emph{canonical Laplace operator $\Delta_{zw}$ on $\Omega$} which is defined by
	\begin{equation}\label{eq:Laplacezw}
		\Delta_{zw}
		F(z,w)
		\coloneqq
		4(1-zw)^2
		\del_z \del_w
		F(z,w)\,,\qquad (z,w)\in\Omega
	\end{equation}
	for $F\in\mathcal{H}(\Omega)$:
	the Laplacian $\Delta_{zw}$ is closely related to the \emph{spherical and hyperbolic Laplacians}
	\begin{align}\label{eq:sphericalLaplaceIntro}
		\Delta_{\widehat{\C}} f(z):=\left(1+|z|^2 \right)^2 \partial_z\partial_{\cc{z}} f(z) \, , \quad z \in \C \, ,\\
		\Delta_{\D} f(z):=\left(1-|z|^2 \right)^2 \partial_z\partial_{\cc{z}} f(z) \, , \quad z\in \D \,,\label{eq:hyperbolicLaplaceIntro}
	\end{align}
	acting on $C^\infty(\widehat{\C})$ resp.\ $C^\infty(\D)$. More precisely, given a holomorphic eigenfunction of $\Delta_{zw}$, its ``shadows'' on the (rotated) diagonal are smooth eigenfunctions of $\Delta_{\widehat{\C}}$ and $\Delta_{\D}$. Motivated by work of Rudin \cite{Rudin84} on the structure of the eigenspaces of $\Delta_{\D}$ with regard to invariance under unit disk automorphisms, Heins, Roth and the author \cite{HeinsMouchaRoth3} investigate the function-theoretical approach described above. Using real methods, Rudin gave a classification of the invariant eigenspaces of $\Delta_{\D}$ by distinguishing between exceptional and non-exceptional eigenvalues. In \cite{HeinsMouchaRoth3} it is shown that in the exceptional case there always exists a finite dimensional subspace such that the eigenfunctions in this particular space admit a ``holomorphic extension'' to $\Omega$, i.e.\ belong to $\mathcal{A}(\D)$. In fact, this ``extension'' is natural in the sense that $\Omega$ is indeed the maximal domain where these functions are holomorphic (see \cite[Def.~2.4]{HeinsMouchaRoth3}). The crucial tool for this purpose is the identification of specific eigenfunctions, so-called \emph{Poisson Fourier modes}, which were already implicitly used in \cite{Rudin84}.

	\medskip

	The purpose of the present paper is a deeper analysis of $\mathcal{H}(\Omega)$ and its offspring, $\mathcal{A}(\widehat{\C})$ and $\mathcal{A}(\D)$. We deal with the following two aspects: first, we identify the Poisson Fourier modes as building blocks of $\mathcal{H}(\Omega)$. Second, based on the first aspect, we provide a complex analytic framework for the spectral theory of the spherical Laplacian with an emphasis on the three-dimensional spherical harmonics. Altogether, our results give an intrinsic description of $\mathcal{A}(\widehat{\C})$ and $\mathcal{A}(\D)$ in terms of the exceptional (in the sense of Rudin \cite{Rudin84}) eigenspaces of the natural Laplacians on $\widehat{\C}$ and $\D$: our results link the spaces $\mathcal{A}(\widehat{\C})$ and $\mathcal{A}(\D)$ precisely with the set of exceptional eigenfunctions singled out by Rudin as well as with the space of three-dimensional spherical harmonics.

	\medskip

	More specifically, the first objective of this paper is to show that the Poisson Fourier modes form a Schauder basis of $\mathcal{H}(\Omega)$, or, in other words, $\mathcal{H}(\Omega)$ admits a Schauder basis consisting of exceptional eigenfunctions. Our proof will use the fact that another Schauder basis of $\mathcal{H}(\Omega)$ is already known (see \cite[Th.~3.16]{KrausRothSchoetzWaldmann2019} and \cite[Prop.~6.4]{SchmittSchoetz2022}, see also \cite{HeinsMouchaRoth1} for a more conceptual approach). This enables us to perform a ``change of Schauder basis'' whose key ingredient is the verification of a curious combinatorial identity. For the latter, some rather deep results about hypergeometric functions based on identities obtained by Whipple \cite{whipple} will be required.

	\medskip

	The second objective of this paper is a function-theoretic description of the spectral theory of $\Delta_{\widehat{\C}}$ by developing a version of \emph{complexified spherical harmonics} based on $\mathcal{H}(\Omega)$. This is motivated by two observations: first, the (point) spectrum of $\Delta_{zw}$ and $\Delta_{\widehat{\C}}$ equal each other.\footnote{Here, we interpret $\Delta_{zw}$ as an operator acting on $\mathcal{H}(\Omega)$. Then, given any eigenfunction $f\in C^\infty(\widehat{\C})$ of $\Delta_{\widehat{\C}}$ for the eigenvalue $4m(m+1)$, where $m$ is a non-negative integer, there is a unique eigenfunction $F\in\mathcal{H}(\Omega)$ of $\Delta_{zw}$ to the same eigenvalue such that $F$ restricted to $\rotdiagonal$ equals $f$, see \cite[Th. 2.5]{HeinsMouchaRoth3}. As operator on $\mathcal{H}(\D\times\D)$, the (point) spectrum of $\Delta_{zw}$ is $\C$ which is in accordance with $\Delta_\D$ and allows for a similar result, see \cite[Th. 2.2]{HeinsMouchaRoth3}.}	Second, $\Omega$ is biholomorphically equivalent to the \emph{complex two-sphere}
	\begin{equation}\label{eq:complextwosphere}
		\mathbb{S}_\C^2 := \left\{\mathbf{z}=(z_1,z_2,z_3) \in \C^3 \,\big|\, z_1^2 + z_2^2 + z_3^2 = 1\right\}
	\end{equation}
	whereas the rotated diagonal $\rotdiagonal$, which, in our picture, encodes the Riemann sphere, translates to
	the \emph{real two-sphere}
	\[\mathbb{S}_\R^2 := \left\{\mathbf{x}=(x_1,x_2,x_3) \in \R^3 \,\big|\, x_1^2 + x_2^2 + x_3^2 = 1\right\} \,.\]
	Since this is the natural domain of the space of classical spherical harmonics in three dimensions, we will be able to identify this space with $\mathcal{A}(\widehat{\C})$. Furthermore, the invariance property of $\Delta_{zw}$ with respect to automorphisms on $\Omega$ induced by M\"obius transformations will allow us to include (complexified) zonal harmonics in our approach.

	\medskip

	This paper is organized as follows: first, in Section~\ref{sec:notation} we provide some notations. In Section~\ref{sec:PFM} we introduce the Poisson Fourier modes (PFM) and formulate our first main result, Theorem~\ref{thm:SchauderbasisIntro}, which states that the PFM form a Schauder basis of $\mathcal{H}(\Omega)$. Section~\ref{sec:Jacobi} connects the PFM to Jacobi polynomials and their orthogonality relations and Section~\ref{sec:PFMSchauderBasis} gives a proof of Theorem~\ref{thm:SchauderbasisIntro}. Then, we proceed with developing complexified spherical harmonics in Section~\ref{sec:CSH}. We describe the sphere model of $\Omega$ in Section~\ref{sub:complexsphere}, give a short introduction to classical spherical harmonics in Section~\ref{sub:RSHtoCSH} which then allows us to achieve our second goal: define and investigate complexified spherical harmonics in Section~\ref{sub:CSH}. Further, Section~\ref{sub:zonalharmonic} introduces the M\"obius group of $\Omega$ in Section~\ref{sec:pullbackformula} and describes the $\Omega$-approach to zonal harmonics in Section \ref{sub:zonal}. Finally, Section~\ref{sec:coefficients} deals with the verification of the combinatorial identity crucial for the proof of Theorem~\ref{thm:SchauderbasisIntro}.

	\section{Notation}\label{sec:notation}

	We denote the open unit disk by $\D:=\{z\in\C\,:\,|z|<1\}$, the punctured plane by $\C^* \coloneqq \C \setminus \{0\}$, the Riemann sphere by $\widehat{\C}$ and the punctured sphere by $\widehat{\C}^* \coloneqq \widehat{\C} \setminus \{0\}$. Moreover, we write $\N \coloneqq \{1, 2, \ldots\}$ for the set of positive integers, $\N_0 \coloneqq \N \cup \{0\}$ and $\Z$ for the set of all integers. For an open subset $U$ of $\widehat{\C}$ or $\widehat{\C}^2$, we write $\partial U$ for its boundary and $\cc{U}$ for its closure. The set of all infinitely (real) differentiable functions  ${f:U\to\C}$ is denoted $C^\infty(U)$, the set of all holomorphic functions ${F:U\to\C}$ is denoted $\mathcal{H}(U)$ and $\Aut(U)$ is the set of all biholomorphic mappings or automorphisms $T : U \to U$. Further, we write $\partial_z$ and $\partial_w$ resp.\ $\partial_{\cc{z}}$ for the complex Wirtinger derivatives for functions of two resp.\ one complex variable. Finally, note that the set $\Omega$ from \eqref{eq:Omega} is a complex manifold of complex dimension $2$ and an open submanifold of $\widehat{\C}^2$.

	Further, we will use the \emph{hypergeometric function} $\F$, which is defined as the power series
	\begin{equation}\label{eq:hyp2F1}
		\Fz{a,b}{c}{t}:=\sum_{k=0}^\infty\frac{(a)_k(b)_k}{(c)_k}\frac{t^k}{k!}\,,\qquad \quad t\in\D\,,\;a,b,c\in\C,\;-c\not\in\N_0
	\end{equation}
	where $(a)_k:=a\cdot(a+1)\cdots(a+k-1)$	denotes the \emph{(rising) Pochhammer symbol}.

	\section{Poisson Fourier modes}\label{sec:PFM}

	We now introduce the functions of interest of the present work:
	\begin{definition}
		Let $m\in\N_0$ and $n\in\Z$. The \emph{$n$-th Poisson Fourier mode (PFM) of its $(-m)$-th power} is the function $	P_n^{-m}:\Omega\rightarrow\C$ given by
		\begin{equation}
			\label{eq:PoissonFourier}
		 P_n^{-m}(z,w)
			:=(-1)^n\binom{m}{|n|}(1-zw)^{-m}\Fz{-m,|n|-m}{|n|+1}{zw}\cdot\begin{cases}
				w^n\quad &\textrm{if }n\geq0\,,\\
				z^{|n|}\quad &\textrm{if }n\leq0\,.
			\end{cases}
		\end{equation}
		Here, we define $\binom{m}{|n|}:=0$, if $|n|>m$.
	\end{definition}
	\begin{remark}\label{rem:PFMIntro}
		\begin{enumerate}[(a)]
			\item The name stems from the characterization as Fourier coefficients of powers of a generalized Poisson kernel, i.e. using \cite[Sec. 2.5.1, Formula (10), p. 81]{Erdelyi1953} it holds that
			\begin{equation}
				P_n^{-m}(z,w)
				=
				\frac{1}{2\pi}
				\int\limits_0^{2\pi}
				\left(\frac{1-zw}{(1-ze^{-it})(1-we^{it})}\right)^{-m}
				e^{-int}
				\, dt \, ,\qquad z,w\in\D\, .
			\end{equation}
			In fact, this characterization serves as the definition of the PFM in \cite[Def. 4.1]{HeinsMouchaRoth3}. However, for our purpose the interpretation of $P_n^{-m}$ as hypergeometric function is more important. Further note that in \cite{HeinsMouchaRoth3} PFM of arbitrary complex powers $\mu$ are considered. However, we do not need this generality.

			\item Note that if $|n| > m$, then $P_{n}^{-m} = 0$. Otherwise, using \eqref{eq:hyp2F1}, $P_{n}^{-m}$ is a rational function:
			\begin{equation}
				\label{eq:PoissonFourierExplicit}
				P_{n}^{-m}(z,w)
				=
				(-1)^n
				\sum_{k=0}^{m-n}
				\binom{m}{k + n}
				\binom{m}{k}
				\frac{z^{k} w^{k+n}}{(1 - z w)^m}
				=
				P_{-n}^{-m}(w,z)\qquad (n\geq0)\,.
			\end{equation}

			\item Since the PFM are symmetric in the sense that $P_{n}^{-m}(z,w)=P_{-n}^{-m}(w,z)$, we can simplify our proofs in the following: we will often prove identities for $P^{-m}_{n}$ only which then implies the corresponding result for $P^{-m}_{-n}$.
		\end{enumerate}
	\end{remark}
	The specific form of the $\F$ function in \eqref{eq:PoissonFourier} allows us to relate the PFM to the \emph{Jacobi polynomials} (we will make this precise in Section \ref{sec:Jacobi}) which are orthogonal polynomials. As a consequence, the PFM are mutually orthogonal (in sense of Proposition~\ref{prop:Jacobiorthogonality}), too.
	\begin{proposition}\label{prop:Jacobiorthogonality}
		Let $m,p\in\N_0$, $n,q\in\Z$. The Poisson Fourier modes fulfil the orthogonality property
		\begin{equation}\label{eq:Jacobiorthogonality}
			\frac{i}{\pi}\int\limits_{\widehat{\C}} P_n^{-m}(z,-\cc{z})\cc{P_q^{-p}(z,-\cc{z})}\frac{dz\, d\cc{z}}{(1+|z|^2)^2}=\frac{\binom{m}{|n|}\delta_{m,p}\delta_{n,q}}{(2m+1)\binom{m+|n|}{|n|}}\, .
		\end{equation}
	Here, $\delta_{m,p}:=1$ resp.\ $\delta_{n,q}:=1$ if $m=p$ resp.\ $n=q$ and zero otherwise.
	\end{proposition}
	 By definition, the functions $P_n^{-m}$ are $-n$-homogeneous, that is
	\begin{equation}\label{eq:nhomogeniety}
		P_n^{-m}
		\left(\xi z, \frac{w}{\xi}\right)
		=
		\xi^{-n}
		P_n^{-m}(z,w)
		\qquad
		z,w \in \D\,, \; \xi\in\partial\D
		\, .
	\end{equation}
	Moreover, the PFM are holomorphic on $\Omega$ by \eqref{eq:PoissonFourierExplicit}. These two properties are essential for the significant role that the PFM play for the spectral theory of the canonical Laplace operator~$\Delta_{zw}$ on $\Omega$ from \eqref{eq:Laplacezw}.	Given $\lambda\in\C$ we denote the \emph{$\lambda$-eigenspace of $\Delta_{zw}$ on $\Omega$} by
	\begin{equation}\label{eq:eigenspaceIntro}
		X_\lambda(\Omega):=\{F\in\mathcal{H}(\Omega)\,:\, \Delta_{zw}F=\lambda F\}\,.
	\end{equation}
	A computation shows that $P_n^{-m}\in X_{4m(m+1)}(\Omega)$ for every $m\in\N_0$. Moreover, every $(-n)$-homogeneous function in $X_{4m(m+1)}(\Omega)$ is (up to a multiplicative constant) given by $P_n^{-m}$ (see \cite[Th. 4.2 combined with (4.18)]{HeinsMouchaRoth3}), and the set $\{P_n^{-m}\,:\, |n|\leq m\}$ forms a basis of $X_{4m(m+1)}(\Omega)$ (see \cite[Th. 8.1]{HeinsMouchaRoth3}). In combination with Proposition \ref{prop:Jacobiorthogonality} this implies the following result.
	\begin{corollary}\label{cor:eigenspaceHilbertspaceIntro}
			Define the map $\langle\cdot{,}\cdot\rangle_\Omega:\mathcal{H}(\Omega)\times \mathcal{H}(\Omega)$ by
			\begin{equation}\label{eq:scalarproductonOmegaIntro}
				\langle F,G\rangle_\Omega:=\frac{i}{\pi}\int\limits_{\widehat{\C}} F(z,-\cc{z})\cc{G(z,-\cc{z})}\frac{dz\,d\cc{z}}{(1+|z|^2)^2}\,.
			\end{equation}
			Let $m\in\N_0$. Then $(X_{4m(m+1)}(\Omega),\langle\cdot{,}\cdot\rangle_\Omega)$ is a Hilbert space with orthogonal basis $ {\{P_{n}^{-m}\,:\, |n|\leq m\}}$.
		\end{corollary}
		\begin{remark}\label{rem:scalarproductIntro}
			By our identification of $\widehat{\C}$ with the rotated diagonal $\rotdiagonal\subseteq\Omega$ (see \eqref{eq:rotdiagonal}) we see that \eqref{eq:scalarproductonOmegaIntro} is well-defined as integral of continuous functions over a compact set. Further note that \eqref{eq:scalarproductonOmegaIntro} already defines an inner product on $\mathcal{H}(\Omega)$. This will follow from Theorem \ref{thm:SchauderbasisIntro} below. Alternatively, one can argue directly with a variant of the identity principle (\cite[p.~18]{Range}): if a holomorphic function $F : U \to \C$ on a domain $U \subseteq \C^2$ containing a point of the form $(z,-\cc{z})$ vanishes on $\rotdiagonal\cap U$, then $F\equiv0$ on $U$.
	\end{remark}
	In Remark 15 in \cite{HeinsMouchaRoth3} it is explained that the closure (w.r.t.\ the compact open topology of $\Omega$) of the linear hull of all PFM equals $\mathcal{H}(\Omega)$. Having this in mind, the first main goal of this paper shows that even more is true:
	\begin{theorem}\label{thm:SchauderbasisIntro}
		The set $\{P_{n}^{-m} \, : \, m\in\N_0,\, n\in\Z,\,|n|\leq m\}$ is a Schauder basis of $\mathcal{H}(\Omega)$. In particular, each $F \in \mathcal{H}(\Omega)$ has a unique representation as
		\begin{equation}\label{eq:PFMasSchauder}
			F(z,w)=\sum \limits_{m=0}^{\infty}\sum\limits_{n=-m}^{m} c_{n,m} P_n^{-m}(z,w) \, .
		\end{equation}
		This series converges absolutely and locally uniformly in $\Omega$, and the Schauder coefficients $c_{\pm n,m}$ ($n\geq0$) of $F$ are given by
		\begin{equation}\label{eq:Schaudercoefficient}
			c_{\pm n,m}=(2m+1)\binom{m+n}{n}\binom{m}{n}^{-1}\langle F,P_{\pm n}^{-m}\rangle_\Omega\,.
		\end{equation}
	\end{theorem}
	\begin{remark}\phantomsection\label{rem:PFMSchauderbasisIntro}
		\begin{enumerate}[(a)]
			\item We can interpret Theorem \ref{thm:SchauderbasisIntro} as a spectral decomposition of $\mathcal{H}(\Omega)$ in terms of eigenfunctions of $\Delta_{zw}$.  In this sense, the PFM form a natural Schauder basis of $\mathcal{H}(\Omega)$.

			\item In view of Proposition \ref{prop:Jacobiorthogonality} the PFM are an ``orthogonal Schauder basis''. Moreover, Theorem \ref{thm:SchauderbasisIntro} gives another justification for the inner product $\langle\cdot,\cdot\rangle_\Omega$ on $\mathcal{H}(\Omega)$ in \eqref{eq:scalarproductonOmegaIntro}, see Remark \ref{rem:scalarproductIntro}.
		\end{enumerate}
	\end{remark}
		Theorem \ref{thm:SchauderbasisIntro} immediately implies for the spaces $\mathcal{A}(\widehat{\C})$ and $
		\mathcal{A}(\D)$ from \eqref{eq:starproductC} and \eqref{eq:starproductD}:
		\begin{corollary}\label{cor:starproductspaces}
			\begin{enumerate}[(a)]
				\item The set $\{{P_{n}^{-m}\vert}_{\rotdiagonal} \, : \, m\in\N_0,\, n\in\Z,\,| n|\leq m\}$ is a Schauder basis of $\mathcal{A}(\widehat{\C})$.
				\item The set $\{{P_{n}^{-m}\vert}_{\diagonal} \, : \, m\in\N_0,\, n\in\Z,\,|n|\leq m\}$ is a Schauder basis of $\mathcal{A}(\D)$.
			\end{enumerate}
	\end{corollary}

	\section{Proof of Proposition \ref{prop:Jacobiorthogonality}: Poisson Fourier modes and orthogonality}\label{sec:Jacobi}

	The \emph{Jacobi polynomials} of parameters $k\in\N_0$, $\alpha,\beta\in\C$ are defined by
		\begin{equation}\label{eq:Jacobi}
			\Jacobi_k^{(\alpha,\beta)}(x):=
			\binom{\alpha+k}{k}\Fz{-k,k+\alpha+\beta+1}{\alpha+1}{\frac{1-x}{2}} ,\qquad x\in[-1,1]\, .
		\end{equation}
	These functions are related to the PFM as follows:
		\begin{lemma}\label{lem:Jacobipolynomials}
			Let $m\in\N_0$, $n\in\Z$ and $|n|\leq m$. Then
		\begin{equation}\label{eq:PoissonisJacobi}
				P_{n}^{-m}(z,w)=\frac{(-1)^m}{(1-zw)^{|n|}}\Jacobi_{ m-|n|}^{(|n|,|n|)}\left(\frac{zw+1}{zw-1}\right)\cdot\begin{cases}
					w^n\quad & \mathrm{if }\, \,  n\geq0\,,\\
					z^{|n|}  \quad & \mathrm{if }\, \, n\leq0\,.
				\end{cases}
		\end{equation}
		\end{lemma}
		\begin{proof}
			Let $n\geq0$. Using a transformation formula for $\F$ functions, see \cite[Eq. 15.3.4]{abramowitz1984}, we have
			\begin{equation}
				P_n^{-m}(z,w)=\binom{m}{n}
				(-1)^n w^n
				\Fz{-m,1+m}{n+1}{\frac{zw}{zw-1}}\,.
			\end{equation}
			Hence,
			\begin{align}\label{eq:Jacobi1}
				P_{n}^{-m}(z,w)&=\binom{m}{n}\binom{m+n}{n}^{-1}(-1)^n w^n\Jacobi_m^{(n,-n)}\left(\frac{1+zw}{1-zw}\right)\, .
			\end{align}
			The Jacobi polynomials fulfil the symmetry property $\Jacobi_k^{(\alpha,\beta)}(x)=(-1)^k\Jacobi_k^{(\beta,\alpha)}(-x)$ as well as
			\begin{equation*}
				\binom{k}{\ell}\Jacobi_k^{(-\ell,\beta)}(x)=\binom{k+\beta}{\ell}\left(\frac{x-1}{2}\right)^\ell\Jacobi_{k-\ell}^{(\ell,\beta)}(x)
			\end{equation*}
			for all $\ell=1,2,...,k$, see \cite[p.~59, 64]{szego75}. Inserting both identities into \eqref{eq:Jacobi1} yields \eqref{eq:PoissonisJacobi}.
		\end{proof}
		The Jacobi polynomials are called orthogonal polynomials due to the following property:
		\begin{equation}\label{eq:Jacobiintegral}
			\int\limits_{-1}^1(1-x)^\alpha(1+x)^\beta\Jacobi_k^{(\alpha,\beta)}(x)\Jacobi_\ell^{(\alpha,\beta)}(x)\, dx=\mathcal{A}_{\alpha,\beta,k}\delta_{k,\ell}
		\end{equation}
		with $\alpha,\beta>-1$, see \cite[Eq.~22.2.1]{abramowitz1984}. Here, $\delta_{k,\ell}:=1$ if $k=\ell$ and $\delta_{k,\ell}:=0$ otherwise, and
		\begin{equation*}\label{eq:Jacobiintegralcoefficient}
			\mathcal{A}_{\alpha,\beta,k}:=\frac{2^{\alpha+\beta+1}\Gamma(k+\alpha+1)\Gamma(k+\beta+1)}{k!(2k+\alpha+\beta+1)\Gamma(k+\alpha+\beta+1)} \, .
		\end{equation*}
		\begin{proof}[Proof of Proposition \ref{prop:Jacobiorthogonality}]
			Let $n,q\geq0$. Using polar coordinates and Lemma \ref{lem:Jacobipolynomials} we can rewrite the left-hand side of \eqref{eq:Jacobiorthogonality} and compute
			\begin{align*}
				&\int\limits_0^{\infty}\frac{1}{2\pi}\int\limits_0^{2\pi} P_n^{-m}(e^{it}\sqrt{s},-e^{-it}\sqrt{s})\cc{P_q^{-p}(e^{it}\sqrt{s},-e^{-it}\sqrt{s})}\frac{dtds}{(1+s)^2} \\
				&\overset{\text{Lem. }\ref{lem:Jacobipolynomials}}{=}\int\limits_0^{\infty}\frac{1}{2\pi}\int\limits_0^{2\pi}(-1)^{m+p}e^{-i(n-q)t}\left(\frac{\sqrt{s}}{1+s}\right)^{n+q}\Jacobi_{m-n}^{(n,n)}\left(\frac{s-1}{s+1}\right)\Jacobi_{p-q}^{(q,q)}\left(\frac{s-1}{s+1}\right)\frac{dtds}{(1+s)^2} \\
				&=(-1)^{m+p}\int\limits_0^{\infty}\underbrace{\frac{1}{2\pi}\int\limits_0^{2\pi}e^{-i(n-q)t}\,dt}_{=\delta_{n,q}}\left(\frac{\sqrt{s}}{1+s}\right)^{n+q}\Jacobi_{m-n}^{(n,n)}\left(\frac{s-1}{s+1}\right)\Jacobi_{p-q}^{(q,q)}\left(\frac{s-1}{s+1}\right)\frac{ds}{(1+s)^2} \\
				&=(-1)^{m+p}\delta_{n,q}\int\limits_0^{\infty}\left(\frac{s}{(1+s)^2}\right)^{n}\Jacobi_{m-n}^{(n,n)}\left(\frac{s-1}{s+1}\right)\Jacobi_{p-n}^{(n,n)}\left(\frac{s-1}{s+1}\right)\frac{ds}{(1+s)^2} \\
				&\overset{x=\tfrac{s-1}{s+1}}{=}(-1)^{m+p}\frac{\delta_{n,q}}{2}\int\limits_{-1}^1\frac{(1+x)^n(1-x)^n}{4^n}\Jacobi_{m-n}^{(n,n)}\left(x\right)\Jacobi_{p-n}^{(n,n)}\left(x\right)\,dx\\
				&\overset{\eqref{eq:Jacobiintegral}}{=}\frac{\mathcal{A}_{n,n,m-n}}{4^n}\frac{\delta_{n,q}\delta_{m,p}}{2}\\
				&=\frac{1}{2m+1}\binom{m}{n}\binom{m+n}{n}^{-1}\delta_{n,q}\delta_{m,p}\, .
			\end{align*}
			For $n,q<0$ only the global signs of the exponential functions in the second and third line change. Similarly, one shows that the integral always vanishes if the subscripts of the PFM in \eqref{eq:Jacobiorthogonality} have different signs.
		\end{proof}

			\section{Proof of Theorem \ref{thm:SchauderbasisIntro}: Poisson Fourier modes as Schauder basis of $\boldsymbol{\mathcal{H}(\Omega)}$}\label{sec:PFMSchauderBasis}\

			The proof of Theorem \ref{thm:SchauderbasisIntro} is divided into several steps because there are some technical considerations involved. However, as this mainly concerns the existence part, we give a proof of uniqueness right away.

			\begin{proof}[Proof of Theorem \ref{thm:SchauderbasisIntro}: uniqueness]\

				By linearity it suffices to show that $H:=\sum_{m=0}^{\infty}\sum_{n=-m}^{m}c_{n,m}P_{n}^{-m}=0$ implies $c_{n,m}=0$ for all $n,m$. Applying Proposition \ref{prop:Jacobiorthogonality} yields
				\begin{align*}
					0=\frac{i}{\pi}\int\limits_{\widehat{\C}}H(z,-\cc{z})\cc{P_n^{-m}(z,-\cc{z})}\frac{dz\, d\cc{z}}{(1+|z|^2)^2}&=\sum_{p=0}^{\infty}\sum_{q=-p}^{p}c_{q,p}
					\frac{i}{\pi}\int\limits_{\widehat{\C}}P_{q}^{-p}(z,-\cc{z})\cc{P_n^{-m}(z,-\cc{z})}\frac{dz\, d\cc{z}}{(1+|z|^2)^2}\\
					&=\sum_{p=0}^{\infty}\sum_{q=-p}^{p}c_{q,p}
					\frac{1}{2m+1}\binom{m}{n}\binom{m+n}{n}^{-1}\delta_{m,p}\delta_{n,q}\\
					&=\frac{1}{2m+1}\binom{m}{n}\binom{m+n}{n}^{-1}c_{n,m} \, .
				\end{align*}
				Hence $c_{n,m}=0$. Note that the computation also shows \eqref{eq:Schaudercoefficient}.
			\end{proof}

		In \cite{HeinsMouchaRoth1} the structure of $\mathcal{H}(\Omega)$ is investigated and a Schauder basis is identified.
			\begin{lemma}[Corollary 4.8 in \cite{HeinsMouchaRoth1}] \label{lem:OldSchauder}
				The set $(f_{p,q})_{p,q \in \N_0}$ of functions where
				\begin{equation*}
					f_{p,q}(z,w)=\frac{z^pw^q}{\left(1-zw\right)^{\max\{p,q\}}}\, ,\quad  p,q\in\N_0  \, ,
				\end{equation*}
				is a Schauder basis of $\mathcal{H}(\Omega)$. In particular, each $F \in \mathcal{H}(\Omega)$ has a unique representation as
				\begin{equation}\label{eq:OldSchauder}
					F(z,w)=\sum \limits_{p,q=0}^{\infty} b_{p,q} f_{p,q}(z,w) \, .
				\end{equation}
				This series converges absolutely and locally uniformly in $\Omega$, and the Schauder coefficients $b_{p,q}$ of $F$ are given by
				\begin{equation} \label{eq:OldSchauderCoeff}
					b_{p,q}= \begin{cases}
						\displaystyle   -\frac{1}{4 \pi^2} \int\limits_{\partial \D} \int
						\limits_{\partial \D} F\left(z, \frac{w}{1+zw} \right) \frac{dzdw}{z^{p+1}
							w^{q+1}} & p<q \, , \\[2mm]
						\displaystyle -\frac{1}{4 \pi^2} \int \limits_{\partial \D} \int
						\limits_{\partial \D} F\left( \frac{z}{1+zw} ,w\right) \frac{dzdw}{z^{p+1}
							w^{q+1}} & p \ge q\, .
					\end{cases}
				\end{equation}
			\end{lemma}
			Note that similar to the PFM the functions $f_{p,q}$ also have a symmetry property, namely $f_{p,q}(z,w)=f_{q,p}(w,z)$. The following result will not be needed in the remainder of this paper. However, it is interesting in its own right: we already know by definition of a Schauder basis, that every PFM can be expressed using the functions $(f_{p,q})_{p,q \in \N_0}$, but, in fact, every PFM can be expressed using only finitely many of the functions $(f_{p,q})_{p,q \in \N_0}$.
			\begin{lemma}\label{lem:PFMwithOldSchauder}
				Let $m \in \N_0$, $n\in\Z$ and $|n|\leq m$. We have
				\begin{equation}\label{eq:PFMwithOldSchauder}
					P_{n}^{-m}=(-1)^n\sum_{j=0}^{m-|n|}\binom{m}{|n|+j}\binom{m+|n|+j}{j}\cdot\begin{cases}
						f_{j,j+n} \quad & \mathrm{if }\, \, n\geq0\,,\\
						f_{j+|n|,j}  \quad & \mathrm{if }\, \, n\leq0\,.
					\end{cases}
				\end{equation}
			\end{lemma}
			\begin{proof}
				Let $n\geq0$ and fix $(z,w)\in\Omega\cap(\C\times\C)$. In view of the global factor $(1-zw)^{-m}$ in \eqref{eq:PoissonFourierExplicit}, we rewrite each of the $f_{j,j+n}$ as
				\begin{equation*}
					f_{j,j+n}(z,w)
					=
					\frac{z^j w^{j+n}}{(1-zw)^{m}}
					(1-zw)^{m-n-j}
					=
					\frac{1}{(1-zw)^{m}}
					\sum_{k=0}^{m-n-j}
					\binom{m-n-j}{k}
					(-1)^k
					z^{j+k} w^{j+n+k}.
				\end{equation*}
				Dropping this common prefactor, \eqref{eq:PFMwithOldSchauder} is equivalent to
				\begin{equation*}
					\label{eq:BasisChangeProof}
					\sum_{j=0}^{m-n}
					\binom{m}{j}
					\binom{m}{n+j}
					z^j w^{j + n}
					=
					\sum_{j=0}^{m-n}
					c_{j,m,n}
					\sum_{k=0}^{m-n-j}
					\binom{m-n-j}{k}
					(-1)^k
					z^{j+k} w^{j+n+k}\,,
				\end{equation*}
				where $c_{j,m,n} = \binom{m}{n+j} \binom{m+n+j}{j}$. Using the linear independence of the polynomials $z^j w^{j+n}$ thus yields a system of linear equations for the coefficients $c_{j,m,n}$. Comparing coefficients for $z^\ell w^{\ell+n}$ leads to the identity
				\begin{equation*}
					\binom{m}{\ell}\binom{m}{n+\ell}
					=
					\sum_{j=0}^\ell
					c_{j,m,n}
					\binom{m-n-j}{\ell-j}
					(-1)^{\ell-j}\, .
				\end{equation*}
				Our choice of $c_{j,m,n}$ indeed satisfies this identity: first, we rewrite
				\begin{align*}
					\sum_{j=0}^{\ell}
					c_{j,m,n}
					\binom{m-n-j}{\ell-j}
					(-1)^{\ell-j}
					&=
					\sum_{j=0}^{\ell}
					\binom{m}{n+j}
					\binom{m+n+j}{j}
					\binom{m-n-j}{\ell-j}
					(-1)^{\ell-j} \\
					&=
					\binom{m}{n+\ell}
					\sum_{j=0}^{\ell}
					\binom{m+n+j}{j}
					\binom{n+\ell}{n+j}
					(-1)^{\ell-j}.
				\end{align*}
				Thus, it remains to show that the sum on the right-hand side equals $\binom{m}{\ell}$. Note that we can replace the sum by a hypergeometric $\F$ function. Moreover, investing the \emph{Chu-Vandermonde identity} \cite[Cor.~2.2.3]{askey}
				\begin{equation*}
					\Fz{-k,b}{c}{1}=\frac{(c-b)_k}{(c)_k}\,, \qquad k\in\N_0\,,\; b,c\in\C\,, \; -c\not\in \N_0
				\end{equation*}
				yields
				\begin{align*}
					\sum_{j=0}^{\ell}
					\binom{m+n+j}{j}
					\binom{n+\ell}{n+j}
					(-1)^{\ell-j}
					&=
					(-1)^{\ell}
					\binom{n+\ell}{n}
					\Fz{
						-\ell, m+n+1}{n+1}{1} \\
					&=
					(-1)^{\ell}
					\binom{n+\ell}{n}
					\frac{(-m)_{\ell}}{(n+1)_{\ell}}
					=
					\binom{m}{\ell}\, .\qedhere
				\end{align*}
			\end{proof}

			An essential observation that allows us to prove the existence part of Theorem \ref{thm:SchauderbasisIntro} is that Lemma \ref{lem:PFMwithOldSchauder} has a converse.

			\begin{lemma}\label{lem:OldSchauderasPFM}
				Let $n,m \in \N_0$, $n\leq m$. For all $(z,w)\in\Omega$ we have
					\begin{align}\label{eq:OldSchauderasPFM}
						f_{n,m}(z,w)&=(-1)^{m-n}
						\sum_{s=0}^{n}
						a_{s,m,n}P_{m-n}^{-(m-s)}(z,w)=f_{m,n}(w,z)
					\end{align}
				where
				\begin{equation}\label{eq:OldSchauderasPFMcoefficients}
					a_{s,m,n}=(-1)^s\binom{m}{s}\binom{2m-s}{n}^{-1}\frac{2m-2s+1}{2m-s+1}\,.
				\end{equation}
			\end{lemma}

			\begin{proof}
				The denominator degree of each term of the sum in \eqref{eq:OldSchauderasPFM} differs. Therefore, we rewrite
				\begin{align*}
					P_{m-n}^{-(m-s)}&=(-1)^{m-n}\sum_{k=0}^{n-s}\binom{m-s}{k+m-n}\binom{m-s}{k}\frac{z^kw^{k+m-n}}{(1-zw)^{m-s}}\\
					&=(-1)^{m-n}\sum_{k=0}^{n-s}\binom{m-s}{k+m-n}\binom{m-s}{k}\frac{z^kw^{k+m-n}}{(1-zw)^{m}}(1-zw)^{s}\\
					&=(-1)^{m-n}\sum_{k=0}^{n-s}\binom{m-s}{k+m-n}\binom{m-s}{k}\frac{z^kw^{k+m-n}}{(1-zw)^{m}}\sum_{t=0}^s(-1)^t\binom{s}{t}(zw)^{t}\, .
				\end{align*}
				Hence, we can drop the factor $w^{m-n}(1-zw)^{-m}$ in \eqref{eq:OldSchauderasPFM}, and we have to prove
				\begin{align*}
					(zw)^n=
					\sum_{s=0}^{n}\sum_{k=0}^{n-s}\sum_{t=0}^s
					(-1)^{s+t}\binom{m}{s}\binom{s}{t}\binom{2m-s}{n}^{-1}\binom{m-s}{k+m-n}\binom{m-s}{k}\frac{2m-2s+1}{2m-s+1}(zw)^{k+t} \, .
				\end{align*}
				Now, we reorder the right-hand side
				\begin{align*}
					(zw)^n=\sum_{s=0}^n\sum_{d=0}^n\sum_{t=0}^{d}	(-1)^{s+t}\binom{m}{s}\binom{s}{t}\binom{2m-s}{n}^{-1}\binom{m-s}{d-t+m-n}\binom{m-s}{d-t}\frac{2m-2s+1}{2m-s+1}(zw)^{d} \, .
				\end{align*}
				Thus, it suffices to show
				\begin{align*}\label{eq:combinatorialID}
					\delta_{n,d}=\sum_{s=0}^n\sum_{t=0}^{d}	(-1)^{s+t}\binom{m}{s}\binom{s}{t}\binom{2m-s}{n}^{-1}\binom{m-s}{d-t+m-n}\binom{m-s}{d-t}\frac{2m-2s+1}{2m-s+1} \, .\tag{CID}
				\end{align*}
				Since \eqref{eq:combinatorialID} is ``just'' a combinatorial identity, we postpone its proof to Section \ref{sec:coefficients}, because the techniques needed for its proof are not required elsewhere in this section.
			\end{proof}

			\begin{corollary}\label{cor:coefficientsleq1}
				The absolute values of the coefficients $a_{s,m,n}$ from \eqref{eq:OldSchauderasPFMcoefficients} are less than or equal to 1.
			\end{corollary}
			\begin{proof}
				For $d=n$ we know from \eqref{eq:combinatorialID} that
				\begin{align*}
					1&=\sum_{s=0}^n\sum_{t=0}^{n}(-1)^{s+t}\binom{m}{s}\binom{s}{t}\binom{2m-s}{n}^{-1}\binom{m-s}{m-t}\binom{m-s}{n-t}\frac{2m-2s+1}{2m-s+1}\\
					&=\sum_{s=0}^n\binom{m}{s}\binom{2m-s}{n}^{-1}\binom{m-s}{n-s}\frac{2m-2s+1}{2m-s+1}=\sum_{s=0}^{n}\binom{m-s}{n-s}|a_{s,m,n}| \, .
				\end{align*}
				Thus, as a sum of non-negative scalars and with
				$\binom{m-s}{n-s}\geq1$ it follows that $|a_{s,m,n}|\leq1$ for all $s=0,...,n$.
			\end{proof}

			Now we can turn to the proof of Theorem \ref{thm:SchauderbasisIntro}. Our strategy is as follows: we use the Schauder representation of $F$ with respect to the ``old'' Schauder basis $(f_{m,n})_{m,n\in\N_0}$ given by \eqref{eq:OldSchauder}, i.e.\
			\begin{equation*}
				F(z,w)=\sum_{m,n=0}^\infty b_{m,n}f_{m,n}(z,w)=\sum_{m=0}^\infty\left[\sum_{n=0}^{m} b_{m,n}f_{m,n}(z,w)+\sum_{n=1}^{m-1} b_{n,m}f_{n,m}(z,w)\right]\, .
			\end{equation*}
			Here, we can replace every $f_{m,n}$ function with a sum of PFM by Lemma \ref{lem:OldSchauderasPFM} which yields
			\begin{align}\label{eq:absolutconvergentseries}
				F(z,w)=\sum_{m=0}^\infty\Bigg[\sum_{n=0}^{m} (-1)^{m-n}b_{m,n}\sum_{s=0}^{n}a_{s,m,n}&P_{n-m}^{-(m-s)}(z,w)\\
				+&\sum_{n=1}^{m-1} (-1)^{m-n}b_{n,m}\sum_{s=0}^{n}a_{s,m,n}P_{m-n}^{-(m-s)}(z,w)\Bigg]\, .\nonumber
			\end{align}
			Thus, we show \eqref{eq:PFMasSchauder} if we show absolute and uniform convergence on compact subsets of the series \eqref{eq:absolutconvergentseries}, because this fact allows us to interchange the order of summation. In the following, given a set $K\subseteq\Omega$ we write $\|F\|_K:=\sup\{|F(z)|\,:\, z\in K\}$.

			\begin{proof}[Proof of Theorem \ref{thm:SchauderbasisIntro}: existence]\

				Let $K\subseteq\Omega$ be compact. Without loss of generality, we can assume that there exists $R>1$ such that $K$ is contained in the intersection of $\Omega$ with one of the sets $K_i$, $i=a,b,c,d$, where
					\begin{align*}
						\mathrm{(a)}\,\, K_a&:=\cc{K_R(0)}\times \cc{K_R(0)}\subseteq\C^2\\
						\mathrm{(b)}\,\, K_b&:=\left(\widehat{\C}\setminus K_R(0)\right)\times\left(\widehat{\C}\setminus K_R(0)\right)\subseteq\widehat{\C}^*\times\widehat{\C}^*\\
						\mathrm{(c)}\,\, K_c&:=\cc{K_R(0)}\times\left(\widehat{\C}\setminus K_R(0)\right)\subseteq\C\times\widehat{\C}^*\\
						\mathrm{(d)}\,\, K_d&:=\left(\widehat{\C}\setminus K_R(0)\right)\times\cc{K_R(0)}\subseteq\widehat{\C}^*\times\C\,.
					\end{align*}
					Here, ${K_R(0):=\{z\in\C\,:\,|z|<R\}}$. To see that this is true note that we can cover $K$ with a finite collection of sets of the form $K_i^\circ$, $i=a,b,c,d$, where $K_i^\circ$ denotes the interior of $K_i$.

				Using \eqref{eq:absolutconvergentseries} we obtain			\begin{align}\label{eq:auxSchaudereistence}
					\|F\|_{K}&\leq\sum_{m=0}^\infty\left[\sum_{n=0}^{m}\sum_{s=0}^{n} |b_{m,n}||a_{s,m,n}|\|P_{n-m}^{-(m-s)}\|_{K}+\sum_{n=1}^{m-1}\sum_{s=0}^{n} |b_{n,m}||a_{s,m,n}|\|P_{m-n}^{-(m-s)}\|_{K}\right]\nonumber\\
					&\overset{\text{Corollary} \ref{cor:coefficientsleq1}}{\leq}\sum_{m=0}^\infty\left[\sum_{n=0}^{m}\sum_{s=0}^{n} |b_{m,n}|\|P_{n-m}^{-(m-s)}\|_{K}+\sum_{n=1}^{m-1}\sum_{s=0}^{n} |b_{n,m}|\|P_{m-n}^{-(m-s)}\|_{K}\right]\, .
				\end{align}
				In order to show absolute and uniform convergence of \eqref{eq:absolutconvergentseries} on $K$, we need to estimate the absolute values of the coefficients $b_{n,m}$ and the values of the PFM $P_{m-n}^{-(m-s)}$ on $K$.

				Let $r_1,r_2>0$. Cauchy's integral formula applied to \eqref{eq:OldSchauderCoeff} implies for $n< m$ that
				\begin{align*}
					|b_{n,m}|&=\left\vert\frac{-1}{4\pi^2}\int\limits_{\partial\D}\int\limits_{\partial\D}F\left(z,\frac{w}{1+zw}\right)\frac{dzdw}{z^{n+1}w^{m+1}}\right\vert\\
					&=\left\vert\frac{1}{4\pi^2}\int\limits_{|z|=r_1}\int\limits_{|w|=r_2}F\left(z,\frac{w}{1+zw}\right)\frac{dzdw}{z^{n+1}w^{m+1}}\right\vert\\
					&\leq \frac{1}{r_1^{n}r_2^{m}}\sup\limits_{|z|=r_1,\, |w|=r_2}\left\vert F\left(z,\frac{w}{1+zw}\right)\right\vert \, .
				\end{align*}
				Analogously, it holds for $n\leq m$ that
				\begin{align*}
					|b_{m,n}|\leq \frac{1}{r_1^{m}r_2^{n}}\sup\limits_{|z|=r_1,\, |w|=r_2}\left\vert F\left(\frac{z}{1+zw},w\right)\right\vert \, .
				\end{align*}
				Note that in both cases we used that the functions
				\[(z,w)\mapsto F\left(z,\frac{w}{1+zw}\right)\quad\text{and}\quad(z,w)\mapsto F\left(\frac{z}{1+zw},w\right) \]
				define holomorphic functions on $\C^2$. Since $F\in\mathcal{H}(\Omega)$, it follows that $F$ is continuous on $\Omega$ and attains its maximum modulus on every compact set. Thus, there is some $M>0$ (depending on $r_1$ and $r_2$) such that
				\begin{equation}\label{eq:bestimate}
					|b_{n,m}|\leq \frac{1}{r_1^{n}r_2^{m}}\cdot M\quad\text{and}\quad|b_{m,n}|\leq \frac{1}{r_1^{m}r_2^{n}}\cdot M \, .
				\end{equation}
				\begin{enumerate}[(a)]
					\item Assume $K\subseteq K_a$. Since $K$ is compact, we know that
					\[\delta_K:=\inf\{|1-zw|\, : \, (z,w)\in K\}>0\,.\]
					Thus, also $\delta:=\min\{\delta_K,4\}>0$. Recall that we can express $P_{n-m}^{-(m-s)}$ as a rational function by \eqref{eq:PoissonFourierExplicit}, i.e.\
					\begin{equation}
						P_{n-m}^{-(m-s)}(z,w)
						=
						(-1)^n
						\sum_{k=0}^{n-s}
						\binom{m-s}{k + m-n}
						\binom{m-s}{k}
						\frac{z^{k+m-n} w^{k}}{(1 - z w)^{m-s}}\,.
					\end{equation}
					Hence, we compute
					\begin{align}\label{eq:Poissonestimate}
						\left\|P_{n-m}^{-(m-s)}\right\|_{K}
						&\leq\sup_{(z,w)\in K} \sum_{k=0}^{n-s}
						\binom{m-s}{k+m-n}
						\binom{m-s}{k}\nonumber
						\left\vert\frac{z^{k+m-n} w^{k}}{(1-zw)^{m-s}}\right\vert
						\\
						&\leq \sup_{(z,w)\in K} \sum_{k=0}^{n-s}
						\binom{m-s}{k+m-n}
						\binom{m-s}{k}
						\frac{\left\vert z^{k+m-n} w^{k}\right\vert}{\delta^{m-s}}\nonumber
						\\
						&\leq \sup_{(z,w)\in K} \sum_{k=0}^{n-s}
						\binom{m-s}{k+m-n}
						\binom{m-s}{k}
						\frac{R^{2k+m-n}}{\delta^{m-s}}\nonumber
						\\
						&\leq\frac{4^{m-s}R^{m-n}}{\delta^{m-s}}\sum_{k=0}^{n-s}(R^2)^k =\frac{4^{m-s}R^{m-n}}{\delta^{m-s}}\frac{1-(R^2)^{n-s+1}}{1-R^2}
					\end{align}
					and the same estimate holds for $\left\|P_{m-n}^{-(m-s)}\right\|_{K}$. Together, \eqref{eq:bestimate} and \eqref{eq:Poissonestimate} imply for \eqref{eq:auxSchaudereistence}
					\begin{align*}\label{eq:auxSchaudereistence2}
						&\sum_{m=0}^\infty\left[\sum_{n=0}^{m}\sum_{s=0}^{n} |b_{m,n}|\|P_{n-m}^{-(m-s)}\|_{K}+\sum_{n=1}^{m-1}\sum_{s=0}^{n} |b_{n,m}|\|P_{m-n}^{-(m-s)}\|_{K}\right]\nonumber
						\\
						&\leq\sum_{m=0}^\infty\left[\sum_{n=0}^{m}\sum_{s=0}^{n} \frac{M}{r_1^{m}r_2^{n}}\frac{4^{m-s}R^{m-n}}{\delta^{m-s}}\frac{1-(R^2)^{n-s+1}}{1-R^2}+\sum_{n=1}^{m-1}\sum_{s=0}^{n} \frac{M}{r_1^{n}r_2^{m}}\frac{4^{m-s}R^{m-n}}{\delta^{m-s}}\frac{1-(R^2)^{n-s+1}}{1-R^2}\right]\nonumber
						\\
						&\leq M\sum_{m=0}^\infty\Bigg[\sum_{n=0}^{m} \left(\frac{4R}{r_1\delta}\right)^m\frac{1-(R^2)^{n+1}}{1-R^2}\frac{1}{r_2^n}\sum_{s=0}^{n}\left(\frac{\delta}{4}\right)^s\nonumber\\
						&\hspace{7cm}+\sum_{n=1}^{m-1}\left(\frac{4R}{r_2\delta}\right)^m\frac{1-(R^2)^{n+1}}{1-R^2}\frac{1}{r_1^n}\sum_{s=0}^{n}\left(\frac{\delta}{4}\right)^s\Bigg]\nonumber
						\\
						&\leq M\sum_{m=0}^\infty\left[\sum_{n=0}^{m} \left(\frac{4R}{r_1\delta}\right)^m\frac{1-(R^2)^{n+1}}{1-R^2}\frac{n+1}{r_2^n}+\sum_{n=1}^{m-1}\left(\frac{4R}{r_2\delta}\right)^m\frac{1-(R^2)^{n+1}}{1-R^2}\frac{n+1}{r_1^n}\right]\,.
					\end{align*}

					Our assumption $R>1$ implies
						\begin{align*}
							\sum_{m=0}^\infty&\left[\sum_{n=0}^{m}\sum_{s=0}^{n} |b_{m,n}|\|P_{n-m}^{-(m-s)}\|_{K}+\sum_{n=1}^{m-1}\sum_{s=0}^{n} |b_{n,m}|\|P_{m-n}^{-(m-s)}\|_{K}\right]
							\\
							&\leq \frac{MR^2}{R^2-1}\left[\sum_{m=0}^\infty \left(\frac{4R}{r_1\delta}\right)^m\sum_{n=0}^{\infty}\left(\frac{R^2}{r_2}\right)^n(n+1)+\sum_{m=0}^\infty\left(\frac{4R}{r_2\delta}\right)^m\sum_{n=0}^{\infty}\left(\frac{R^2}{r_1}\right)^n(n+1)\right]\,.
						\end{align*}
						Now choose $r_1,r_2>\max\{4R/\delta,R^2\}>1$ to guarantee convergence.

					\item Assume $K\subseteq K_b$. Then $\widetilde{K}:=\{(1/w,1/z)\,:\, (z,w)\in K\}$ is a compact set and contained in {$K_R(0)\times K_R(0)$}. Now, it follows immediately from \eqref{eq:PoissonFourierExplicit} that
					\begin{equation}
						P_n^{-m}(z,w)=(-1)^mP_n^{-m}\left(\frac{1}{w},\frac{1}{z}\right)
					\end{equation}
					for all $(z,w)\in\Omega$. Hence,
					\[\left\Vert P_{\pm(n-m)}^{-(m-s)}\right\Vert_K=\left\Vert P_{\pm(n-m)}^{-(m-s)}\right\Vert_{\widetilde{K}}\, .\]
					Therefore, we can apply Case (a).

					\item Assume $K\subseteq K_c$. Again, since $K$ is compact, we have
					\[\delta_{K}:=\inf\{|1-zw|\, : \, (z,w)\in K\}>0\,.\]
					It follows that
					\[\delta_{K}':=\inf\left\{\left\vert\frac{1}{w}-z\right\vert\,:\,(z,w)\in K\right\}>0\, .\]
					Moreover, $\widetilde{K}=\{(1/w,1/z)\,:\, (z,w)\in K\}$ is compact, and thus also
					\[\delta_{\widetilde{K}}:=\inf\left\{\left\vert 1-\frac{1}{zw}\right\vert\,:\,(z,w)\in K\right\}>0\,.\]
					Define $\delta=\min\{\delta_K,\delta_K',\delta_{\widetilde{K}},4\}$. Using
					\[\left\vert\frac{zw}{1-zw}\right\vert=\left\vert\frac{1}{1-\frac{1}{zw}}\right\vert\leq\frac{1}{\delta_{\widetilde{K}}}\]
					we compute
					\begin{align}\label{eq:Poissonestimatec1}
						\left\|P_{n-m}^{-(m-s)}\right\|_{K}
						&\leq\sup_{(z,w)\in K} \sum_{k=0}^{n-s}
						\binom{m-s}{k+m-n}
						\binom{m-s}{k}
						\left\vert\frac{z^{k+m-n} w^{k}}{(1-zw)^{m-s}}\right\vert\nonumber
						\\
						&\leq \sup_{(z,w)\in K} \sum_{k=0}^{n-s}
						\binom{m-s}{k+m-n}
						\binom{m-s}{k}
						\cdot\begin{cases}
							\left\vert\frac{(zw)^{m-s}}{(1-zw)^{m-s}}\right\vert\cdot R^{m-n},\quad \text{if } |zw|> 1\\
							\left\vert\frac{1}{(1-zw)^{m-s}}\right\vert\cdot R^{m-n},\quad \text{if } |zw|\leq 1\\
						\end{cases}\nonumber
						\\
						&\leq \sup_{(z,w)\in K} \sum_{k=0}^{n-s}
						\binom{m-s}{k+m-n}
						\binom{m-s}{k}
						\cdot\begin{cases}
							\frac{1}{\delta_{\widetilde{K}}^{m-s}} R^{m-n},\quad \text{if } |zw|> 1\\
							\frac{1}{\delta_K^{m-s}} R^{m-n},\quad \text{if } |zw|\leq 1\\
						\end{cases}\nonumber
						\\
						&\leq \sum_{k=0}^{n-s}\frac{4^{m-s}R^{m-n}}{\delta^{m-s}}\nonumber
						\\
						&\leq n\frac{4^{m-s}R^{m-n}}{\delta^{m-s}}\, .
					\end{align}

					Similarly, we compute
					\begin{align}\label{eq:Poissonestimatec2}
						\left\|P_{m-n}^{-(m-s)}\right\|_{K}&\leq\sup_{(z,w)\in K} \sum_{k=0}^{n-s}
						\binom{m-s}{k+m-n}
						\binom{m-s}{k}
						\left\vert\frac{z^{k} w^{k+m-n}}{(1-zw)^{m-s}}\right\vert\nonumber
						\\
						&\leq\sup_{(z,w)\in K} \sum_{k=0}^{n-s}
						\binom{m-s}{k+m-n}
						\binom{m-s}{k}
						\left\vert\frac{z^{k} }{w^{n-s-k}(1/w-z)^{m-s}}\right\vert\nonumber
						\\
						&\leq\sup_{(z,w)\in K} \sum_{k=0}^{n-s}
						\binom{m-s}{k+m-n}
						\binom{m-s}{k}
						R^{n-s}\frac{1}{\delta^{m-s}}\nonumber
						\\
						&\leq \sum_{k=0}^{n-s}\frac{4^{m-s}R^{n-s}}{\delta^{m-s}}\nonumber
						\\
						&\leq n\frac{4^{m-s}R^{n-s}}{\delta^{m-s}}\, .
					\end{align}
					Inserting \eqref{eq:bestimate}, \eqref{eq:Poissonestimatec1} and \eqref{eq:Poissonestimatec2} in \eqref{eq:auxSchaudereistence} yields
					\begin{align*}
						\sum_{m=0}^\infty&\left[\sum_{n=0}^{m}\sum_{s=0}^{n} |b_{m,n}|\|P_{n-m}^{-(m-s)}\|_{K}+\sum_{n=1}^{m-1}\sum_{s=0}^{n} |b_{n,m}|\|P_{m-n}^{-(m-s)}\|_{K}\right]
						\\
						&\leq\sum_{m=0}^\infty\left[\sum_{n=0}^{m}\sum_{s=0}^{n} \frac{M}{r_1^mr_2^n}n\frac{4^{m-s}R^{m-n}}{\delta^{m-s}}+\sum_{n=1}^{m-1}\sum_{s=0}^{n} \frac{M}{r_1^nr_2^m}n\frac{4^{m-s}R^{n-s}}{\delta^{m-s}}\right]
						\\
						&\leq M\sum_{m=0}^\infty\left[\sum_{n=0}^{m} \left(\frac{4R}{r_1\delta}\right)^{m}\frac{n}{(r_2R)^n}\sum_{s=0}^{n}\left(\frac{\delta}{4}\right)^s+\sum_{n=1}^{m-1}\left(\frac{4}{r_2\delta}\right)^{m}\left(\frac{R}{r_1}\right)^n n\sum_{s=0}^{n}\left(\frac{\delta}{4}\right)^s\right]
						\\
						&\leq M\sum_{m=0}^\infty\left[\sum_{n=0}^{m} \left(\frac{4R}{r_1\delta}\right)^{m}\frac{n(n+1)}{(r_2R)^n}+\sum_{n=1}^{m-1}\left(\frac{4}{r_2\delta}\right)^{m}\left(\frac{R}{r_1}\right)^n n(n+1)\right]
						\\
						&\leq M\left[\sum_{m=0}^\infty \left(\frac{4R}{r_1\delta}\right)^{m}\sum_{n=0}^{\infty}\frac{n(n+1)}{(r_2R)^n}+\sum_{m=0}^\infty\left(\frac{4}{r_2\delta}\right)^{m}\sum_{n=0}^{\infty}\left(\frac{R}{r_1}\right)^n n(n+1)\right]\,.
					\end{align*}
					Choose $r_1>\max\{4R/\delta,R\}$ and $r_2>\max\{4/\delta,1/R\}$ to guarantee convergence.

					\item Assume $K\subseteq K_d$. Then we can argue similarly as in Case (c).\qedhere
				\end{enumerate}
			\end{proof}

			\begin{remark}\label{rem:PFMSchauderbasis}
				By slight modification of the previous proof we obtain an \emph{orthonormal} Schauder basis of $\mathcal{H}(\Omega)$: in view of Corollary \ref{cor:eigenspaceHilbertspaceIntro} defining
						\begin{equation}
							Y_n^{-m}:=\sqrt{2m+1}\binom{m}{|n|}^{-1/2}\binom{-m-1}{|n|}^{1/2}P_n^{-m}
						\end{equation}
						yields mutually \emph{orthonormal} functions with respect to $\langle\cdot{,}\cdot\rangle_\Omega$. Therefore, we obtain an orthonormal set ${\{Y_n^{-m}\,:\, m\in\N_0,\, n\in\Z,\, |n|\leq m\}}$. This set is a Schauder basis of $\mathcal{H}(\Omega)$, too, because a closer look at the proof of Corollary \ref{cor:coefficientsleq1} reveals that the absolute values of the Schauder coefficients $a_{s,m,n}$ from Lemma \ref{lem:OldSchauderasPFM} are in fact less than or equal to $\binom{m-s}{n-s}$. Using this, one can copy the proof of Theorem \ref{thm:SchauderbasisIntro}.
			\end{remark}

			\section{Complex spherical harmonics}\label{sec:CSH}\

		Before we introduce the \emph{complex spherical harmonics} in Section \ref{sub:CSH}, we discuss a different model of $\Omega$ first: the complex two-sphere $\mathbb{S}_\C^2$
		\begin{equation}
			\mathbb{S}_\C^2 := \left\{\mathbf{z}=(z_1,z_2,z_3) \in \C^3 \,\big|\, z_1^2 + z_2^2 + z_3^2 = 1\right\}
		\end{equation}
		 from \eqref{eq:complextwosphere}. After that, in Section \ref{sub:RSHtoCSH}, we give some background about classical spherical harmonics. In general, for a solid introduction to that theory, see e.g.\ \cite[Ch.~5]{axler2001} and \cite[Ch.~3]{simon15}.

			\subsection{The complex two-sphere}\label{sub:complexsphere}\

			\medskip

			The complex two-sphere $\mathbb{S}_\C^2$ is a closed but unbounded subset of $\C^3$, thus not compact. The domain $\Omega$ is biholomorphically equivalent to $\mathbb{S}_\C^2$ by way of the biholomorphic map ${S \colon \Omega \longrightarrow \mathbb{S}_{\C}^2}$,
			\begin{equation}
				\label{eq:Stereographic}
				S(z,w)
				=
				\begin{cases}\displaystyle
					\left( \frac{z - w}{1 - zw}, -i \frac{z + w}{1 - zw}, -\frac{1 + z w}{1 - z w} \right) & \phantom{ if \quad } (z,w) \in \C^2\,,\; zw\not=1\,, \\[4mm]
					\displaystyle  \left(1/z ,i/z, 1 \right) & \text{ if \quad } z \in \widehat{\C}^*\,, \;w=\infty\,,\\[4mm]
					\displaystyle      \left( -1/w, i/w,1 \right) & \phantom{ if \quad } w \in \widehat{\C}^*\,, \;z=\infty\,.
				\end{cases}
			\end{equation}
			Note the difference in our notation between $\mathbf{z}\in\C^3$ and $z\in\widehat{\C}$. The inverse map of  $S \colon \Omega \longrightarrow \mathbb{S}_{\C}^2$ is given by $S^{-1} \colon \mathbb{S}_{\C}^2 \to \Omega$,
			\begin{equation}\label{eq:inversecomplexstereo}
				S^{-1}(z_1,z_2,z_3)
				=
				\begin{cases}\displaystyle
					\left( \frac{z_1+i z_2}{1-z_3}, - \frac{z_1 -i z_2}{1-z_3} \right) & \phantom{ if \quad } z_3 \not=1\,,  \\[4mm]
					\displaystyle  \left(1/z_1,\infty \right) & \text{ if \quad } z_3=1\,, \;z_2=i z_1\,,\\[4mm]
					\displaystyle      \left( \infty, -1/z_1  \right)& \phantom{ if \quad } z_3=1\,, \;z_2=-i z_1\,.
				\end{cases}
			\end{equation}
			It is natural to call $S^{-1} \colon \mathbb{S}_{\C}^2 \to \Omega$ the ``complex stereographic projection'', since if we denote by
			\[  \mathbb{S}_{\R}^2 \coloneqq \mathbb{S}^2_{\C} \cap \R^3 = \big\{ \mathbf{x}=(x_1, x_2, x_3) \in \R^3 \,|\, x_1^2 + x_2^2 + x_2^3 = 1\big\}\]
			the euclidean two-sphere in $\R^3$, then
			\[
			S^{-1}(\mathbf{x})=\left( \frac{x_1+i x_2}{1-x_3}, -\frac{x_1-i x_2}{1-x_3} \right) \quad \text{ for all } \mathbf{x}=(x_1,x_2,x_3) \in \mathbb{S}_{\R}^2 \, \]
			and
			\[ \mathbb{S}^2_{\R} \to \widehat{\C} \, , \qquad \mathbf{x} \mapsto  \frac{x_1+i x_2}{1-x_3}\]
			is the standard stereographic projection of $\mathbb{S}^2_{\R}$ onto $\widehat{\C}$. In particular,
			$S^{-1}$ maps $\mathbb{S}_{\R}^2$ onto the ``rotated diagonal'' $\rotdiagonal$ (see \eqref{eq:rotdiagonal}).
			\begin{remark}\label{rem:hyperboloid}
				Let $H_{\R}^2$ be the image of the hyperboloid ${\{ \mathbf{x}=(x_1,x_2,x_3) \in \R^3 \, : \, -x_1^2-x_2^2+x_3^2=1\}}$ under the biholomorphic map given by $(z_1,z_2,z_3) \mapsto (iz_1,iz_2,z_3)$. Then $S^{-1}$ maps ``half of'' $H_\R^2$ onto the diagonal $\diagonal$. More precisely, the lower half $\{\mathbf{x}=(x_1,x_2,x_3) \in H_{\R}^2  \, : \,  x_3 \le -1\}$
				is mapped onto $\diagonal$ and the upper half $\{\mathbf{x}=(x_1,x_2,x_3) \in H_{\R}^2  \, : \,  x_3 \ge 1\}$
				onto $\{(z, \overline{z}) \, : \, z \in \widehat{\C} \setminus \overline{\D}\}$.
			\end{remark}
			Moreover, for $\mathbf{z}\in\C^3$ with $z_1^2+z_2^2+z_3^2\in\C\setminus(-\infty,0]$ we define the ``complexified length of vectors in $\C^3$'' by
				\begin{equation*}
					\vert\mathbf{z}\vert_\C:=\sqrt{z_1^2+z_2^2+z_3^2}\,.
				\end{equation*}
				Here, the square root is to be understood as the principal branch of the square root with branch cut along the non-positive real axis. It follows that for every such $\mathbf{z}$ we have $\mathbf{z}/\vert\mathbf{z}\vert_\C\in\mathbb{S}_\C^2$.

			\subsection{From real to complex spherical harmonics}\label{sub:RSHtoCSH}\

			\medskip

			When studying the real euclidean Laplace equation in $\R^d$, $d\geq2$, i.e
			\begin{equation}\label{eq:harmonic}
				\Delta_{\R^d}f(\mathbf{x}):=\sum_{k=1}^d\frac{\partial^2}{\partial x_k^2}f(\mathbf{x})=0\, \quad \mathbf{x}=(x_1,\dots,x_d)\in\R^d\, ,
			\end{equation}
			where $f:\R^d\to\C$ is a twice differentiable function, one common approach is to work with spherical harmonics. The name ``spherical'' stems from the fact that spherical harmonics ``live'' on the real $d-1$ sphere, i.e.\ they are functions defined on
			\begin{equation*}
				\mathbb{S}_\R^{d-1} := \left\{\mathbf{x}=(x_1,\dots,x_d) \in \R^d \,\big|\, x_1^2 + \dots + x_d^2 = 1\right\} \, .
			\end{equation*}
			More precisely, for $m\in\N_0$ and an \emph{harmonic} and \emph{$m$-homogeneous} polynomial $p:\R^d\to\C$, i.e.\ $p$ is a polynomial fulfilling \eqref{eq:harmonic} and
			\begin{equation}\label{eq:sphericalhomogeneous} 
				p(r\mathbf{x})=r^m p(\mathbf{x})\, ,\quad\text{for all }r>0\; ,
			\end{equation}
			the restriction $q:=p\vert_{\mathbb{S}^{d-1}_\R}$ is called a (real or classical) \emph{spherical harmonic} of degree $m$.
			\begin{remark}
				Recall \eqref{eq:nhomogeniety} where we identified the PFM as $n$-homogeneous functions. Homogeneity in the sense of \eqref{eq:sphericalhomogeneous} is different from the understanding in \eqref{eq:nhomogeniety}. Throughout this section and the following we will use the term \emph{homogeneous function} for functions fulfilling \eqref{eq:sphericalhomogeneous} because this definition is standard in the theory of (real) spherical harmonics. Moreover, homogeneity in the sense of \eqref{eq:nhomogeniety} will be referred to as \emph{$\Omega$-homogeneity}.
			\end{remark}
			For functions defined on the real sphere, it is useful to work in spherical coordinates by which we mean identifying each $\mathbf{x}\in\R^{d}\setminus\{0\}$ with a tuple $(r,\sigma)\in\R_+\times\mathbb{S}^{d-1}_\R$. This point of view enables one to always construct a harmonic $m$-homogeneous polynomial $p$ to a given spherical harmonic $q$ of degree $m$ by setting $p(r,\sigma)=r^mq(\sigma)$. In particular, if we consider the spherical part of the real euclidean Laplacian, i.e.\ the restriction of $\Delta_{\R^d}$ to $\mathbb{S}_\R^{d-1}$ which we denote by $\Delta_{\mathbb{S}_{\R}^{d-1}}$, then it is an easy calculation to verify the following characterization:
			\begin{equation}\label{eq:realsphericalharmcharacterization}
				\Delta_{\R^d}p=0\quad\iff\quad \Delta_{\mathbb{S}_\R^{d-1}}q=-m(m+d-2)q\, .
			\end{equation}
			Now note that for the special case $d=3$ a different description of $\mathbb{S}_\R^2$ is given by means of the stereographic projection, i.e.\ we can identify the real two-sphere with the Riemann sphere $\widehat{\C}$. It follows immediately that $\Delta_{\mathbb{S}_\R^{2}}$ translates to $\Delta_{\widehat{\C}}$ from \eqref{eq:sphericalLaplaceIntro}. Since there actually is a counterpart of the two-sphere as well as the stereographic projection in our complexified theory, see the previous Section \ref{sub:complexsphere}, we can use this approach together with \eqref{eq:realsphericalharmcharacterization} to define complexified, or in short complex, spherical harmonics.

			\subsection{Complex spherical harmonics}\label{sub:CSH}\

			\medskip

			Let $F \in \mathcal{H}(\Omega)$ and fix $m\in\N_0$. By means of the inverse of the complex stereographic projection we define a holomorphic function by
				\begin{equation}\label{eq:CSHonsphereonly}
					G:\mathbb{S}_\C^2\rightarrow\C\,,\quad G:=F\circ S^{-1}\,.
				\end{equation}
				Using this function we define another holomorphic function
				\begin{equation}\label{eq:complexsphericalharmonic}
					\hat{G}:\{\mathbf{z}\in\C^3\,:\,\mathbf{z}=t\sigma,\, t\in\C,\, \mathrm{Re}\,t>0,\,\sigma\in\mathbb{S}_\C^2\}\rightarrow\C\,,\quad \hat{G}(\mathbf{z})=\vert\mathbf{z}\vert_\C^{m} (F\circ S^{-1})\left(\frac{\mathbf{z}}{\vert\mathbf{z}\vert_\C}\right)\,.
				\end{equation}
				Note that the domain of $\hat{G}$ is biholomorphically equivalent to the cartesian product of the right half plane with $\mathbb{S}_\C^2$. Furthermore, $\hat{G}$ fulfils
				\[\hat{G}(r\sigma)=r^m\hat{G}(\sigma)\quad \text{for all }r>0\,,\]
				and a computation yields
				\begin{equation} \label{eq:sphericalharmonics}
					\Delta_{\C^3} \hat{G}(\mathbf{z}):=\sum_{k=1}^3\frac{\partial^2}{\partial z_k^2}\hat{G}(\mathbf{z})= \vert\mathbf{z}\vert_\C^{2-m} \left(\left( m (m+1)F - \left(1- \, z w \right)^2 \partial_z\partial_w F \right)\circ S^{-1}\right)\left(\frac{\mathbf{z}}{\vert\mathbf{z}\vert_\C}\right) \, .
				\end{equation}
				This identity immediately implies
				\begin{equation*}\label{eq:complexsphericalharmcharacterization}
					\Delta_{\C^3}\hat{G}=0\quad\iff\quad \Delta_{zw}F=4m(m+1)F\, .
				\end{equation*}
				By the definition of the eigenspaces $ X_{4m(m+1)}(\Omega)$ of $\Delta_{zw}$ in \eqref{eq:eigenspaceIntro}, both statements are also equivalent to $F\in X_{4m(m+1)}(\Omega)$. Moreover, in view of Lemma \ref{lem:Jacobipolynomials} we can deduce the following.
				\begin{lemma}\label{lem:inducedharmonicpolynomial}
					Let $m\in\N_0$, $F\in X_{4m(m+1)}(\Omega)$ and define $\hat{G}$ as in \eqref{eq:complexsphericalharmonic}. Then $\hat{G}$ is an $m$-homogeneous polynomial in $\mathbf{z}\in\C^3$. In particular, if $F=P_{\pm n}^{-m}$ for $0\leq n\leq m$, then it holds that
					\begin{align}\label{eq:CSHJacobi}
						\hat{G}\Big\vert_{\mathbb{S}_\C^2}(\mathbf{z})&=(-1)^{m}\left(\frac{\mp z_1+iz_2}{2}\right)^n\Jacobi_{ m-n}^{(n,n)}\left(z_3\right) \, .
					\end{align}
				Here, $\Jacobi_{ m-n}^{(n,n)}$ denotes the Jacobi polynomials from \eqref{eq:Jacobi}.
				\end{lemma}
				\begin{proof}
					Denote the domain of $\hat{G}$ given in \eqref{eq:complexsphericalharmonic} by $\Lambda$. By definition, $\hat{G}$ is $m$-homogeneous and holomorphic on $\Lambda$. In view of the PFM being a basis of the finite dimensional space $X_{4m(m+1)}(\Omega)$ it suffices to show that $\hat{G}$ coincides with a polynomial on ${\Lambda}$ whenever $F=P_{\pm n}^{-m}$. However, since this requires some tedious computations, we only stress the important ideas: one verifies \eqref{eq:CSHJacobi} by rewriting $P_{\pm n}^{-m}$ as a Jacobi polynomial (see Lemma \ref{lem:Jacobipolynomials}) and using the explicit formula for $S^{-1}$. Then by induction over $m$ one can check that
					\[(-1)^{m}|\mathbf{z}|_\C^{m-n}\left(\frac{\mp z_1+iz_2}{2}\right)^n\Jacobi_{ m-n}^{(n,n)}\left(\frac{z_3}{|\mathbf{z}|_\C}\right)\]
					is indeed a polynomial. Here, the recursion formula \cite[Eq.~22.7.1]{abramowitz1984} for Jacobi polynomials is useful.
			\end{proof}
			Therefore, the functions $F$ and $\hat{G}$ fulfil all the properties that we have singled out about (real) spherical harmonics $q$ and the corresponding harmonic homogeneous polynomials $p$ in Section~\ref{sub:RSHtoCSH}. More precisely,
			\begin{enumerate}[(i)]
				\item if $\hat{G}$ is an $m$-homogeneous polynomial in $\C^3$ fulfilling $\Delta_{\C^3}\hat{G}=0$, then $\hat{G}|_{\R^3}$ is an \mbox{$m$-homogeneous} polynomial in $\R^3$ fulfilling $\Delta_{\R^3}\hat{G}=0$.
				\item if $F\in X_{4m(m+1)}(\Omega)$, then $f(z):=F(z,-\overline{z})$ fulfils $\Delta_{\widehat{\C}}f=4m(m+1)f$. Moreover, $S^{-1}$ maps $\mathbb{S}_{\R}^2$ bijectively onto the ``rotated diagonal'' $\rotdiagonal$ as was noted in Section \ref{sub:complexsphere}.
			\end{enumerate}
			This justifies the following definition.

			\begin{definition}[Complex spherical harmonics]
				Let $m\in\N_0$ and $F\in X_{4m(m+1)}(\Omega)$. Let $G_F:=G\in \mathcal{H}(\mathbb{S}_\C^2)$ where $G$ is the function induced by $F$ in \eqref{eq:CSHonsphereonly}. We define the \emph{space of complex spherical harmonics of degree $m$} as the linear hull of all $G_F$, i.e.
				\begin{equation*}
					\mathrm{CSH}(m):=\mathrm{span}\{G_F \, : \, F\in X_{4m(m+1)}(\Omega)\}\,.
				\end{equation*}
				Every $Q\in\mathrm{CSH}(m)$ is called \emph{complex spherical harmonic of degree $m$}. Further, define
				\begin{equation*}
					\mathrm{CSH}:=\left\{\sum_{k=1}^N c_k F_k\,:\, N\in\N,\, F_k\in \mathrm{CSH}(m)\text{ for some }m\in\N_0,\, c_k\in\C\right\}\,.
				\end{equation*}
				Every function $Q\in\mathrm{CSH}$ is called \emph{complex spherical harmonic}. Moreover, we denote by $\mathrm{CSH}^{-}$ the closure of $\mathrm{CSH}$ with respect to the locally uniform topology on $\mathbb{S}_\C^2$.
			\end{definition}
			Essentially, since $S^{-1}$ is bijective, we can now use our knowledge about $\mathcal{H}(\Omega)$ and, in particular, the PFM in order to understand the complex spherical harmonics. Let $m\in\N_0$, {$| n|\leq m$}. Motivated by \cite{HeinsMouchaRoth3} where it is shown that the PFM play a special role for the eigenspace theory of $\Delta_{zw}$ we introduce their spherical counterpart.
			\begin{definition}
				Let $m\in\N_0$ and $n\in\Z$. The \emph{$n$-th spherical Poisson Fourier mode (sPFM) of its $(-m)$-th power} is the function $Q_n^{-m}:\mathbb{S}_\C^2\to\C$ given by
				\begin{equation*}\label{eq:sphericalPoissonFouier}
					Q_n^{-m}=P_n^{-m}\circ S^{-1} \; .
				\end{equation*}
			\end{definition}
			Since the PFM $\{P_n^{-m}\,:\,|n|\leq m\}$ span $X_{4m(m+1)}(\Omega)$, see \cite[Th. 8.1]{HeinsMouchaRoth3}, we conclude that $\mathrm{CSH}(m)$ is spanned by the sPFM $\{Q_n^{-m}\,:\,|n|\leq m\}$ and has dimension $2m+1$. Note that this matches precisely the dimension of the space of real spherical harmonics of degree $m$ in $\R^3$, see \cite[Cor.~3.5.6]{simon15}. Further, we can again translate the orthogonality of the Jacobi polynomials to our setting.
			\begin{corollary}\label{cor:Jacobiorthogonalityspherical}
				The sPFM $Q_n^{-m}$ fulfil the following orthogonality property
				\begin{equation*}\label{eq:JacobiorthogonalitySphere}
					\int\limits_{\mathbb{S}_\R^2} Q_n^{-m}(\mathbf{x})\cc{Q_q^{-p}(\mathbf{x})}\,d\mathbf{x}=\binom{m}{n}\binom{m+n}{n}^{-1}\frac{\delta_{m,p}\delta_{n,q}}{2m+1}\, .
				\end{equation*}
				Here, $d\mathbf{x}$ denotes the normalized area measure on $\mathbb{S}_\R^2$.
			\end{corollary}
			\begin{proof}
				This is the spherical version of Proposition \ref{prop:Jacobiorthogonality} since $S(\rotdiagonal)=\mathbb{S}_\R^2$.
			\end{proof}
			This result allows us to explicitly describe the space(s) of complex spherical harmonics.
			\begin{theorem}\phantomsection\label{prop:SchauderSphere}
				\begin{itemize}
					\item[(a)] The set $\{Q_{n}^{-m} \, : \, m\in\N_0,\, n\in\Z,\, |n|\leq m\}$ is a Schauder basis of $\mathrm{CSH}^-$. In particular, $\mathrm{CSH}^-=\mathcal{H}(\mathbb{S}_\C^2)$ and each $Q \in \mathcal{H}(\mathbb{S}_\C^2)$ has a unique representation as
					\begin{equation*}\label{eq:sPFMasSchauder}
						Q(\mathbf{z})=\sum \limits_{m=0}^{\infty}\sum\limits_{n=-m}^{m} c_{n,m} Q_n^{-m}(\mathbf{z}) \, .
					\end{equation*}
					This series converges absolutely and locally uniformly in $\mathbb{S}_\C^2$, and the Schauder coefficients $c_{\pm n,m}$ of $Q$ are given by
					\begin{equation*}\label{eq:SchaudercoefficientSphere}
						c_{\pm n,m}=\frac{2m+1}{2}\binom{m}{n}^{-1}\binom{m+n}{n}\int\limits_{\mathbb{S}_\R^2} Q(\mathbf{x})\cc{Q_{\pm n}^{-m}(\mathbf{x})}\,d\mathbf{x} \, .
					\end{equation*}
					\item[(b)] The map $\langle\cdot{,}\cdot\rangle_{\mathbb{S}_\R^2}\colon\mathcal{H}(\mathbb{S}_\C^2)\times\mathcal{H}(\mathbb{S}_\C^2)\to\C$ defined by
					\begin{equation}\label{eq:innerproduct}
						\langle Q,P\rangle_{\mathbb{S}_\R^2}=\int\limits_{\mathbb{S}_\R^2} Q(\mathbf{x})\cc{P(\mathbf{x})}\,d\mathbf{x}
					\end{equation}
					is an inner product on $\mathcal{H}(\mathbb{S}_\C^2)$. In particular, $(\mathrm{CSH}(m),\langle\cdot{,}\cdot\rangle_{\mathbb{S}_\R^2})$ is a $2m+1$-dimensional Hilbert space for all $m\in\N_0$ and the sPFM $\{Q_{n}^{-m}\,:\, |n|\leq m\}$ form an orthogonal basis. Moreover, the spaces $\mathrm{CSH}(m)$ and $\mathrm{CSH}(p)$, $m,p\in\N_0$, are mutually orthogonal with respect to $\langle\cdot{,}\cdot\rangle_{\mathbb{S}_\R^2}$.
				\end{itemize}
			\end{theorem}
			\begin{proof}
				Part (a) is a consequence of the PFM being a Schauder basis of $\mathcal{H}(\Omega)$ (Theorem \ref{thm:SchauderbasisIntro}), the sPFM being orthogonal (Corollary \ref{cor:Jacobiorthogonalityspherical}) and $S^{-1}$ being a bijection. Part (b) is the spherical version of Corollary \ref{cor:eigenspaceHilbertspaceIntro}. Since $S(\rotdiagonal)=\mathbb{S}_\R^2$, we can again (see Remark \ref{rem:scalarproductIntro}) argue with the identity principle to see that \eqref{eq:innerproduct} indeed defines an inner product on $\mathcal{H}(\mathbb{S}_\C^2)$. The orthogonality properties follow from Corollary \ref{cor:Jacobiorthogonalityspherical}.
			\end{proof}
			Theorem \ref{prop:SchauderSphere} immediately implies for the space $\mathcal{A}(\widehat{\C})$ from \eqref{eq:starproductC}:
				\begin{corollary}
					Let $f\in C^\infty(\widehat{\C})$. Then $f\in\mathcal{A}(\widehat{\C})$ if and only if $f\circ S^{-1}=Q\vert_{\mathbb{S}_\R^2}$ for some $Q\in\mathcal{H}(\mathbb{S}_\C^2)$. In this case, $Q$ is uniquely determined.
			\end{corollary}

			\begin{remark}
				Recall \eqref{eq:complexsphericalharmonic}. By a theorem of Cartan (see e.g.\ \cite[Ch. VIII, Sec. A, Th. 18]{Gunning1965}) 
				every function that is holomorphic on $\mathbb{S}_\C^2$ possesses a (not necessarily unique) holomorphic extension to $\C^3$. Thus, there exists $\widetilde{G} \in \mathcal{H}(\C^3)$ such that its restriction to $\mathbb{S}^2_{\C}$ is $\hat{G}|_{\mathbb{S}^2_{\C}}$.
				Now suppose that $F \in \mathcal{H}(\Omega)$ is an eigenfunction of $\Delta_{zw}$. Then $F$ is given by a linear combination of PFM. In view of Lemma \ref{lem:inducedharmonicpolynomial} and \eqref{eq:sphericalharmonics} this is equivalent to $\hat{G}$ being a homogeneous polynomial hence, in particular, $\hat{G} \in \mathcal{H}(\C^3)$ and $\Delta_{\C^3}\hat{G}=0$. This, in turn, is in accordance with Theorem 1 of Wada \cite{Wada1988} which shows that there is a \text{unique} $\widetilde{G} \in \mathcal{H}(\C^3)$ with $\Delta_{\C^3} \widetilde{G}=0$ and $\widetilde{G}|_{\mathbb{S}^2_{\C}}=F \circ S^{-1}=\hat{G}|_{\mathbb{S}^2_{\C}}$. For the proof, Wada also uses a version of complex spherical harmonics on the so-called Lie sphere, which was mainly developed by Morimoto, see e.g.\ \cite{morimoto1980} and \cite{morimoto1983}. However, Morimoto's theory strongly depends on the Lie sphere which does not show up in our theory.
			\end{remark}
			\begin{remark}
				By Theorem \ref{prop:SchauderSphere} the sPFM form an orthogonal basis of the space $\mathrm{CSH}(m)$. In fact, the sPFM are even a special basis because the additional index $n$, which arose from the $\Omega$-homogeneity \eqref{eq:nhomogeniety}, has a concrete counterpart in the real theory: the classical approach solving the real Laplacian eigenvalue equation for the eigenvalue $4m(m+1)$ leads to the \emph{general Legendre equations}. Solving these differential equations yields the Legendre polynomials of degree $m$ and order $n$. Thus, solutions arising from those polynomials in the real theory are the natural counterpart to our sPFM. For more details on this we recommend \cite[Sec.~7.3]{gallier2020} and the references given therein.
			\end{remark}

			\section{Complex zonal harmonics}\label{sub:zonalharmonic}\

			Consider the matrix group $\mathrm{SO}(3,\R)$ of real orthogonal matrices or, in short, rotations of $\mathbb{S}^2_\R$. Note that $\mathrm{SO}(3,\R)$ consists precisely of the isometries of $\R^3$, i.e.\ the linear transformations which preserve the standard inner product
			\begin{equation*}\label{eq:scalarproduct}
				\langle \mathbf{x},\mathbf{y}\rangle:=\sum_{j=1}^3x_jy_j\,,\qquad \mathbf{x},\mathbf{y}\in\R^3\, .
			\end{equation*}
			These matrices arise naturally in the theory of (real) spherical harmonics because the space of (real) spherical harmonics of fixed degree $m\in\N_0$ is invariant under the action of $\mathrm{SO}(3,\R)$. Moreover, there exists a unique (up to a multiplicative factor) function $\mathrm{z}_m\colon\mathbb{S}^2_\R\times\mathbb{S}_\R^2\to\R$ with the properties that $z_m$ is a (real) spherical harmonic of degree $m$ if either of its arguments is fixed, and such that $z_m$ is invariant under rotations applied to both variables, i.e.\
			\begin{equation}\label{eq:realzonalinvariance}
				z_m(R\mathbf{x},R\mathbf{y})=z_m(\mathbf{x},\mathbf{y}),\qquad R\in\mathrm{SO}(3,\R),\, \mathbf{x},\mathbf{y}\in\mathbb{S}_\R^2\, .
			\end{equation}
			Hence, $\mathrm{z}_m$ depends only on the inner product of its arguments. This function $\mathrm{z}_m$ is called \emph{(real) zonal harmonic of degree $m$}. In order to find a counterpart of $z_m$ in our complexified theory, we will first investigate the counterpart of $\mathrm{SO}(3,\R)$: the ``rotations'' of $\mathbb{S}_\C^2$, that is
			\[\text{SO}(3,\C):=\{R \in \C^{3 \times 3} \, : \, R^TR=I,\, \det R=1\}\,.\]
			Note that for our purpose it will be easier to work in the $\Omega$-setting again. By way of the inverse complex stereographic projection, $\text{SO}(3,\C)$ corresponds to the \emph{M\"obius group of $\Omega$} defined by
			\begin{equation}\label{eq:OmegaAutomorphism}
				\mathcal{M}
				:=
				\bigcup \limits_{ \psi \in \Aut(\widehat{\C})}
				\left\{ (z,w) \mapsto
				T_\psi(z,w):=\left(
				\psi(z),\frac{1}{\psi(1/w)}
				\right)\, , \,
				(z,w) \mapsto
				\left(
				\psi(w),\frac{1}{\psi(1/z)}
				\right) \right\}\, .
			\end{equation}
			We say that $T_\psi\in\mathcal{M}$ is induced by $\psi\in\Aut(\widehat{\C})$. The significance of $\mathcal{M}$ for the spectral theory of $\Delta_{zw}$\footnote{Here, we consider $\Delta_{zw}$ as operator on $\mathcal{H}(\Omega)$. For the significance of $\mathcal{M}\cap\Aut(\D\times\D)$ for  $\Delta_{zw}$ as operator on $\mathcal{H}(\D\times\D)$ see \cite{HeinsMouchaRoth3}.} (and thus for the complex spherical harmonics) is that $\mathcal{M}$ consists precisely of the biholomorphic automorphisms of $\Omega$ satisfying
			\begin{equation}
				\Delta_{zw} \left(F \circ T \right)=\left(\Delta_{zw} F\right) \circ T
			\end{equation}
			for all $F\in\mathcal{H}(\Omega)$, see \cite[Th. 5.2]{HeinsMouchaRoth1}. Therefore, every eigenspace $X_{4m(m+1)}(\Omega)$ is invariant w.r.t.\ precompositions with elements of $\mathcal{M}$. In view of \eqref{eq:realzonalinvariance} our goal is to prove the following result.
			\begin{theorem}\label{thm:zonalIntro}
				Fix $(u,v)\in\Omega$. Let $m\in\N_0$. Then there is a unique (up to a multiplicative factor) function in $X_{4m(m+1)}(\Omega)$ that is invariant with respect to all automorphisms in $\mathcal{M}$ that fix $(u,v)$.
			\end{theorem}
			In Section \ref{sub:zonal} we develop our complexified version of zonal harmonics and, among others, prove Theorem \ref{thm:zonalIntro}. To this end, we collect some basic properties about $\mathcal{M}$ in the next Section~\ref{sec:pullbackformula}. For a first read, since the results in Section \ref{sec:pullbackformula} are rather technical,  we recommand to turn towards Section \ref{sub:zonal} next, and consult Section \ref{sec:pullbackformula} whenever the need arises. Further, for the theory of real zonal harmonics we refer to \cite[pp.~94--97]{axler2001}, \cite[Th.~3.5.9]{simon15} and \cite[Sec.~7.6]{gallier2020}.

			\subsection{Properties of the M\"obius group}\label{sec:pullbackformula}\

			\medskip

			We can characterize $\mathcal{M}$ as follows:
			\begin{lemma}[Lemma 2.2 in \cite{HeinsMouchaRoth2}]
				\label{lem:MoebiusGenerators}
				The group $\mathcal{M}$ is generated by the flip map $\mathcal{F}:(u,v)\mapsto(1/v,1/u)$ and the mappings
				\begin{equation}
					\label{eq:MoebiusGenerators}
					T_{z,w}(u,v)
					\coloneqq
					\left(
					\frac{z-u}{1-wu} ,
					\frac{w-v}{1-zv}
					\right)
					\quad \textrm{and} \quad
					\varrho_{\gamma}(u,v)
					\coloneqq
					\left(
					\gamma u, \frac{1}{\gamma }v
					\right)
					\, , \qquad
					(u,v) \in \Omega
				\end{equation}
				with $(z,w) \in\Omega\cap (\C\times\C)$ and $\gamma \in \C^*$. More precisely, for every $T \in \mathcal{M}$ there exist $z,w \in \C$ and $\gamma \in \C^*$ such that
				\begin{equation*}
					\label{eq:MoebiusGeneratorsExplicit}
					T
					=
					\varrho_\gamma
					\circ
					T_{z,w}\quad\text{or}\quad T
					=
					\varrho_\gamma
					\circ
					T_{z,w}
					\circ
					\mathcal{F}
					\, .
				\end{equation*}
			\end{lemma}
			\begin{remark}
				For $w=\overline{z}\in\D$ in \eqref{eq:MoebiusGenerators} we see that $T_{z,w}$ is induced by a self-inverse automorphisms of $\D$ (not being the identity). Similarly, choosing {$w=-\overline{z}\in\C$}, the automorphism $T_{z,w}$ is induced by a self-inverse rigid motions of $\widehat{\C}$.
			\end{remark}
			Note that we replaced the automorphisms $\Phi_{z,w}$ occurring in \cite{HeinsMouchaRoth2} by the automorphisms ${T_{z,w}=\Phi_{z,w}\circ\varrho_{-1}}$ in \eqref{eq:MoebiusGenerators}. For our purposes this is more convenient because, essentially, the automorphisms $T_{z,w}$ in \eqref{eq:MoebiusGenerators} are precisely the automorphisms of $\Omega$ interchanging a given point $(z,w)\in\Omega\cap(\C\times\C)$ with $(0,0)$ (instead of only sending $(0,0)$ to $(z,w)$). This is the content of the following lemma which is easily verified by calculation. Here, for $z,w\in\widehat{\C}^*$ we define
			\begin{equation}\label{eq:Tzwtilde}
				\widetilde{T}_{z,w}(u,v)=\left(T_{\mathcal{F}(z,w)}\circ\mathcal{F}\right)(u,v)=T_{1/w,1/z}\left(\frac{1}{v},\frac{1}{u}\right),\qquad (u,v)\in\Omega\,.
			\end{equation}
			We will also need the \emph{swap} on $\Omega$, that is the automorphism
			\begin{equation}\label{eq:swap}
				\mathcal{S}(u,v)=(v,u),\qquad (u,v)\in\Omega\,.
			\end{equation}

			\begin{lemma}\label{lem:TzwProperties}
				Let $(z,w) \in \Omega$ and $T \in \mathcal{M}$ such that $T(z,w)=(0,0)$ and $T(0,0)=(z,w)$.
				\begin{itemize}
					\item[(a)] Let $(z,w) \in \C^*\times\C^*$. Then $T=T_{z,w}$ or $T=\widetilde{T}_{z,w}$.
					\item[(b)] Let $(z,w)\in\C\times\C$ such that either $z=0$ or $w=0$. Then $T=T_{z,w}$.
					\item[(c)] Let $(z,w)\in\widehat{\C}^*\times\widehat{\C}^*$ such that either $z=\infty$ or $w=\infty$. Then $T=\widetilde{T}_{z,w}$.
					\item[(d)] If $(z,w)=(0,0)$, then there is $\gamma \in \C^*$ such that either $T=\varrho_\gamma$ or $T=\varrho_\gamma\circ\mathcal{S}$.
					\item[(e)]  If $(z,w)=(\infty,\infty)$, then there is $\gamma \in \C^*$ such that either $T=\varrho_\gamma\circ\mathcal{F}$ or $T=\varrho_\gamma\circ\mathcal{F}\circ\mathcal{S}$.
				\end{itemize}
			\end{lemma}
			The automorphisms $T_{z,w}$ and $\widetilde{T}_{z,w}$ do not form a subgroup of the M\"obius group $\mathcal{M}$ because composing two automorphisms of this type produces an additional rotation $\varrho_\gamma$. The same is true when composing $T_{z,w}$ with an arbitrary automorphism $T\in\mathcal{M}$. We will give a conceptual interpretation of the following result in Remark \ref{rem:scalarproductforM}.

				\begin{lemma}\label{lem:invarianceonOmega}
					Let $(z,w),(u,v)\in\Omega$. Let $T\in\mathcal{M}$ with $T(z,w)=(z',w')$ and $T(u,v)=(u',v')$. Then there is $\gamma_j\in\C^*$ and $\tau_j\in\{0,1\}$, $j=a,b,c$, such that
					\begin{subequations}
						\begin{align}\label{eq:invarianceonOmega}
							\varrho_{\gamma_a}\circ\mathcal{S}^{\tau_a}\circ T_{u,v}(z,w)&=T_{u',v'}(z',w')\,,\qquad\text{if }u,v,u',v'\in\C\,,\\
							\varrho_{\gamma_b}\circ\mathcal{S}^{\tau_b}\circ \widetilde{T}_{u,v}(z,w)&=\widetilde{T}_{u',v'}(z',w')\,,\qquad\text{if }u,v,u',v'\in\widehat{\C}^*\,,\label{eq:invarianceonOmega2}\\
							\varrho_{\gamma_c}\circ\mathcal{S}^{\tau_c}\circ T_{u,v}\label{eq:invarianceonOmega3}(z,w)&=\widetilde{T}_{u',v'}(z',w')\,,\qquad\text{if }u,v\in\C,\,u',v'\in\widehat{\C}^*\,.
						\end{align}
					\end{subequations}
					Here, we write $\mathcal{S}^0(z,w)=(z,w)$ and $\mathcal{S}^1(z,w)=(w,z)$ for $(z,w)\in\Omega$.
				\end{lemma}
				\begin{proof}
					We use Lemma \ref{lem:MoebiusGenerators} to write $T=\varrho_\kappa\circ T_{\alpha,\beta}\circ\mathcal{F}^\sigma$ where $\kappa\in\C^*$, $\alpha,\beta\in\C$ and $\sigma\in\{0,1\}$. For the individual components it holds that
					\begin{alignat*}{3}
						\left(T_{\varrho_\kappa(u,v)}\circ\varrho_\kappa\right)(z,w)&=\left(\varrho_\kappa\circ T_{u,v}\right)(z,w)\,,&&\qquad (z,w)\in\Omega\\
						\left(T_{T_{\alpha,\beta}(u,v)}\circ T_{\alpha,\beta}\right)(z,w)&=\left(\varrho_{\frac{\alpha v-1}{1-\beta u}}\circ T_{u,v}\right)(z,w)\,,&&\qquad (z,w)\in\Omega\,,\; u\neq1/\beta,\, v\neq1/\alpha\\
						\left(T_{T_{\alpha,\beta}\circ\mathcal{F}(u,v)}\circ T_{\alpha,\beta}\circ\mathcal{F}\right)(z,w)&=\left(\varrho_{\frac{\alpha -u}{v-\beta }}\circ\mathcal{S}\circ T_{u,v}\right)(z,w)\,,&&\qquad (z,w)\in\Omega\,,\; u\neq\alpha,\; v\neq\beta\, .
					\end{alignat*}
					Moreover, note that $\varrho_{\kappa}\circ\varrho_{\kappa'}=\varrho_{\kappa\kappa'}$. First assume $\sigma=0$. Further, assume $u\neq1/\beta$, {$v\neq1/\alpha$} since then $(u',v')=T(u,v)\in\Omega\cap(\C\times\C)$. Using the above identities it follows that
					\begin{equation*}
						\left(T_{T(u,v)}\circ T\right)(z,w)=\varrho_{\kappa}\left(\left(T_{T_{\alpha,\beta}(u,v)}\circ T_{\alpha,\beta}\right)(z,w)\right)=\left(\varrho_{\kappa\frac{\alpha v-1}{1-\beta u}}\circ T_{u,v}\right)(z,w)\,
					\end{equation*}
					which is \eqref{eq:invarianceonOmega} for $\tau_a=0$ and $\gamma_a=\kappa(\alpha v-1)/(1-\beta u)$. Now assume $\sigma=1$ and $u\neq\alpha$, $v\neq\beta$ which assures $(u',v')=T(u,v)\in\Omega\cap(\C\times\C)$. Then it follows that
					\begin{equation*}
						\left(T_{T(u,v)}\circ T\right)(z,w)=\varrho_{\kappa}\left(\left(T_{T_{\alpha,\beta}\circ\mathcal{F}(u,v)}\circ T_{\alpha,\beta}\circ\mathcal{F}\right)(z,w)\right)=\left(\varrho_{\kappa\frac{\alpha -u}{v-\beta}}\circ\mathcal{S}\circ T_{u,v}\right)(z,w)
					\end{equation*}
					which is \eqref{eq:invarianceonOmega} for $\tau_a=1$ and $\gamma_a=\kappa(\alpha -u)/(v-\beta)$. The other cases are proven similarly.
				\end{proof}
				\begin{remark}
					The previous proof contains explicit formulas for the scalars $\gamma_i$, $i=a,b,c$. The appearance of the swap $\mathcal{S}$ in the formulas \eqref{eq:invarianceonOmega} and \eqref{eq:invarianceonOmega2} indicates that $\sigma=1$ in the expression of $T$ in the proof of Lemma \ref{lem:invarianceonOmega}. On the contrary, if $\mathcal{S}$ appears in \eqref{eq:invarianceonOmega3}, then $\sigma=0$. Moreover note that for $T=T_{\alpha,\beta}$, with $\alpha,\beta,u,v\in\C$, $u\neq1/\beta$, $v\neq1/\alpha$, identity \eqref{eq:invarianceonOmega} results in
					\begin{equation*}
						\left(T_{T_{\alpha,\beta}(u,v)}\circ T_{\alpha,\beta}\right)(z,w)=\left(\varrho_{\gamma}\circ T_{u,v}\right)(z,w)\quad \text{where}\quad \gamma=\sqrt{\psi_{\alpha,\beta}'(u)/(1/\psi_{\alpha,\beta}(1/v))'}
					\end{equation*}
					which was already noted in \cite[Prop.~3.6]{HeinsMouchaRoth2}.
				\end{remark}
				Recall that we want to develop ``complexified zonal harmonics'' in the following Section \ref{sub:zonal}. Therefore, we need to find a counterpart to \eqref{eq:realzonalinvariance} in the $\Omega$-setting. It is clear that we replace $R\in\mathrm{SO}(3,\R)$ with $T\in\mathcal{M}$. As explained at the beginning of Section \ref{sub:zonalharmonic}, in the classical theory one is interested in the orthogonal matrices which are precisely the linear transformations preserving the standard inner product $\langle\cdot{,}\cdot\rangle$ on $\R$. In fact, $\mathrm{SO}(3,\C)$ also preserves $\langle\cdot{,}\cdot\rangle$. Thus, we are interested in the analogous quantity to $\langle\cdot{,}\cdot\rangle$ on $\Omega$ which turns out to be the \emph{cross ratio}
				\begin{equation}\label{eq:crossratio}
					[u_1,u_2,u_3,u_4]:=\frac{(u_1-u_3)(u_2-u_4)}{(u_1-u_4)(u_2-u_3)}\,,\qquad u_1,u_2,u_3,u_4\in\widehat{\C}\,.
				\end{equation}
				(If either of the points $u_j$, $j=1,2,3,4$, equals $\infty$, then we understand \eqref{eq:crossratio} as the appropriate limit.)  \begin{lemma}\label{lem:scalarproductforautomorphisms}
					Let $(z,w),(u,v),(z',w'),(u',v')\in\Omega$. Then there is $T\in\mathcal{M}$ such that $T(z,w)=(z',w')$ and ${T(u,v)=(u',v')}$ if and only if
					\begin{equation*}
						\left[z,\frac{1}{w},u,\frac{1}{v}\right]=\left[z',\frac{1}{w'},u',\frac{1}{v'}\right]\,.
					\end{equation*}
				\end{lemma}
				\begin{proof}
					The statement is a consequence of the well-known property of the cross ratio that given points $u_j, v_j\in\widehat{\C}$, $j=1,2,3,4$, then
					\[[u_1,u_2,u_3,u_4]=[v_1,v_2,v_3,v_4]\]
					if and only if there is $\psi\in\Aut(\widehat{\C})$ with $\psi(u_j)=v_j$, $j=1,2,3,4$, see \cite[p.~78]{Ahlfors1979}.
				\end{proof}
				\begin{remark}\label{rem:scalarproductforM}
					To see that the cross ratio indeed is the counterpart to $\langle\cdot,\cdot\rangle$ note the following.

					Let $T\in\mathcal{M}$ be as in Lemma \ref{lem:scalarproductforautomorphisms}. This is equivalent to $R:=S\circ T\circ S^{-1}\in\mathrm{SO}(3,\C)$ such that $R(\mathbf{x})=\mathbf{x}'$ and $R(\mathbf{y})=\mathbf{y}'$, where $\mathbf{x},\mathbf{y},\mathbf{x}',\mathbf{y}'$ are the preimages of $(z,w),(u,v),(z',w'),(u',v')$ under $S^{-1}$. By the orthogonality of $R$ this again is equivalent to $\langle \mathbf{x},\mathbf{y}\rangle=\langle R\mathbf{x},R\mathbf{y}\rangle=\langle\mathbf{x}',\mathbf{y}'\rangle$.

					Now assume $u,v,u',v'\in\C$. Translating the above observations to $\Omega$ and using the definition of $S$, see \eqref{eq:Stereographic}, yields
					\begin{align}
						&\phantom{\iff}\langle S(z,w),S(u,v)\rangle=\langle S(z',w'),S(u',v')\rangle\nonumber\\
						&\iff\frac{1+\psi_{u,v}(z)/\psi_{u,v}(1/w)}{1-\psi_{u,v}(z)/\psi_{u,v}(1/w)}=\frac{1+\psi_{u',v'}(z')/\psi_{u',v'}(1/w')}{1-\psi_{u',v'}(z')/\psi_{u',v'}(1/w')}\nonumber\\
						&\iff\psi_{u,v}(z)/\psi_{u,v}(1/w)=\psi_{u',v'}(z')/\psi_{u',v'}(1/w')\\
						&\iff  \left[z,1/w,u,1/v\right]=\left[z',1/w',u',1/v'\right]\label{eq:psirelation} \, .
					\end{align}
					Hence, we can understand Lemma \ref{lem:scalarproductforautomorphisms} as follows: there is $T\in\mathcal{M}$ such that $T(z,w)=(z',w')$ and ${T(u,v)=(u',v')}$ if and only if the product of the components of $T_{u',v'}(z',w')$ equals the product of the components of $T_{u,v}(z,w)$. Here, $T_{u,v}$ can be replaced by $\widetilde{T}_{u,v}$ if both automorphisms are defined, and $T_{u,v}$ needs to be replaced by $\widetilde{T}_{u,v}$ if only the second automorphism is defined. The same applies for $T_{u',v'}$. Note that these observation provides another proof of Lemma \ref{lem:scalarproductforautomorphisms} by means of Lemma~\ref{lem:invarianceonOmega}.
				\end{remark}

				In order to explicitly describe (complex) zonal harmonics we will also need the following result.
				\begin{proposition}[Proposition 7.2 in \cite{HeinsMouchaRoth3}]\label{lem:Pullbackproposition}
					Let $m\in\N_0$ and $T=T_\psi$ or $T=T_\psi\circ\mathcal{S}$ where $\psi\in\Aut(\widehat{\C})$. Then it holds for all $(z,w)\in\Omega$
					\begin{equation*}\label{eq:MoebiusPullback}
						\big(
						P_0^{-m} \circ T
						\big)(z,w)=\sum_{j=-m}^{m}
						\frac{(-m-j)_j}{(m-j+1)_j}P^{-m}_{-j}\left(\psi^{-1}(0),1/\psi^{-1}(\infty)\right)
						P^{-m}_{j}(z,w)\,.    .
					\end{equation*}
				\end{proposition}
				\begin{remark}
					With the help of Proposition \ref{lem:Pullbackproposition} one can show, see \cite[Cor.~8.4]{HeinsMouchaRoth3}, that the set ${\{P_0^{-m}\circ T\,:\, T\in\mathcal{M}\}}$ spans $X_{4m(m+1)}(\Omega)$.
				\end{remark}

			\subsection{Zonal harmonics}\label{sub:zonal}\

			\medskip

			Fix $m\in\N_0$. By Remark \ref{rem:PFMSchauderbasis} the set $\{Y_n^{-m}\,:\, m\in\N_0,\,n\in\Z,\,|n|\leq m\}$ where
			\begin{equation}\label{eq:orthonormalPoissonFourier}
				Y_n^{-m}:=\sqrt{2m+1}\binom{m}{|n|}^{-1/2}\binom{-m-1}{|n|}^{1/2}P_n^{-m}
			\end{equation}
			forms an orthonormal basis of the eigenspace $X_{4m(m+1)}(\Omega)$. Similarly to the real theory we define:
			\begin{definition}\label{def:complexzonal}
				Let $m\in\N_0$ and $Y_n^{-m}$, $n=-m,\ldots,m$, be given by \eqref{eq:orthonormalPoissonFourier}. We define the \emph{complex zonal harmonic ${Z_m:\Omega\times\Omega\to\C}$ of degree $m$} by
				\begin{equation}\label{eq:zonalharmonic}
					Z_m\big((z,w),(u,v)\big):=\sum_{j=-m}^mY_j^{-m}(z,w)Y_{-j}^{-m}(u,v)\,.
				\end{equation}
			\end{definition}
			A closer look at the previous definition reveals that the complex zonal harmonics can be characterized using particular automorphisms of $\mathcal{M}$, namely $T_{u,v}=T_{\psi_{u,v}}$ and $\widetilde{T}_{u,v}=T_{\tilde{\psi}_{u,v}}$ where
			\begin{equation}\label{eq:phizw}
				\psi_{u,v}(z)=\frac{u-z}{1-vz}\quad\text{resp.}\quad\widetilde{\psi}_{u,v}(z)=\frac{u}{v}\frac{1-vz}{u-z},\qquad z\in\widehat{\C} \,
			\end{equation}
			for $u,v\in\C$ resp.\ $u,v\in\widehat{\C}^*$ (see also \eqref{eq:MoebiusGenerators} and \eqref{eq:Tzwtilde}).
			\begin{proposition}\label{prop:zonaliszeroPoissonForier}
				Let $m\in\N_0$ and $(z,w),(u,v)\in\Omega$. Then it holds that
				\begin{subequations}\label{eq:zonaliszeroPoissonForier}
					\begin{alignat}{2}
						\label{eq:zonaliszeroPoissonForier1}Z_m\big((z,w),(u,v)\big)&=(2m+1)\left(P_0^{-m}\circ T_{u,v}\right)(z,w)\qquad&&\text{if }u,v\in\C \, ,\\
						Z_m\big((z,w),(u,v)\big)&=(2m+1)\left(P_0^{-m}\circ \widetilde{T}_{u,v}\right)(z,w)\qquad&&\text{if }u,v\in\widehat{\C}^* \,.\label{eq:zonaliszeroPoissonForiertilde}
					\end{alignat}
				\end{subequations}
			\end{proposition}
			 Note that \eqref{eq:zonaliszeroPoissonForier1} and \eqref{eq:zonaliszeroPoissonForiertilde} coincide for $z,w\in\C^*$ because in this case the product of the components of $T_{u,v}(z,w)$ equals the product of the components of $\widetilde{T}_{u,v}(z,w)$.
			\begin{proof}[Proof of Proposition \ref{prop:zonaliszeroPoissonForier}]
				This is Lemma \ref{lem:Pullbackproposition} for $T=T_{u,v}$ if $u,v\in\C$ or $T=\widetilde{T}_{u,v}$ if $u,v\in\widehat{\C}^*$.
			\end{proof}

			We can now prove the main result of this chapter, which is our counterpart of \eqref{eq:realzonalinvariance} and includes Theorem \ref{thm:zonalIntro}.
			\begin{theorem}\label{thm:zonalinvariance}
				Let $m\in\N_0$ and $Z_m$ be the complex zonal harmonic of degree $m$.
				\begin{itemize}
					\item[(a)] For every $T\in\mathcal{M}$, $(z,w),(u,v)\in\Omega$ we have $Z_m\left(T(z,w),T(u,v)\right)=Z_m\left((z,w),(u,v)\right)$. In particular, $Z_m$ only depends on the vector $T_{u,v}(z,w)$ resp.\ $\widetilde{T}_{u,v}(z,w)$.
					\item[(b)] Fix $(u,v)\in\Omega$. Then $Z^{u,v}_m:=Z_m\left(\cdot,(u,v)\right)$ is the unique (up to a multiplicative factor) function in $X_{4m(m+1)}(\Omega)$ that is invariant with respect to automorphisms in $\mathcal{M}$ that fix $(u,v)$, i.e.\ for every $T\in\mathcal{M}$ such that $T(u,v)=(u,v)$ it holds that $Z_m^{u,v}\circ T=Z_m^{u,v}$.
				\end{itemize}
			\end{theorem}
			\begin{proof}
				\begin{enumerate}[(a)]
					\item The dependence on the vector $T_{u,v}(z,w)$ resp.\ $\widetilde{T}_{u,v}(z,w)$ only was already proven in~\eqref{eq:zonaliszeroPoissonForier}. The other claim is Lemma \ref{lem:scalarproductforautomorphisms} and the $0$-$\Omega$-homogeneity of $P_0^{-m}$ (see~\eqref{eq:nhomogeniety}) because for $(u,v),T(u,v)\in\Omega\cap\C^2$ it holds that
					\begin{align*}
						Z_m\big(T(z,w),T(u,v)\big)&\overset{\eqref{eq:zonaliszeroPoissonForier1}}{=}(2m+1)\left(P_0^{-m}\circ T_{T(u,v)}\right)\big(T(z,w)\big)\\
						&\overset{\phantom{\eqref{eq:zonaliszeroPoissonForier1}}}{=}(2m+1)\left(P_0^{-m}\circ T_{u,v}\right)(z,w)\overset{\eqref{eq:zonaliszeroPoissonForier1}}{=}Z_m\big((z,w),(u,v)\big)\,.
					\end{align*}
					The case $(u,v),T(u,v)\in\Omega\cap(\widehat{\C}^*)^2$ is handled analogously by \eqref{eq:zonaliszeroPoissonForiertilde}.

					\item As a linear combination of elements in $X_{4m(m+1)}(\Omega)$, clearly, $Z^{u,v}_m\in X_{4m(m+1)}(\Omega)$. For the other property let us first consider $(u,v)=(0,0)$. By Lemma \ref{lem:TzwProperties}(d) the automorphisms $T\in\mathcal{M}$ that fix $(0,0)$ are precisely of the form $T=\varrho_\gamma$ or $\varrho_\gamma\circ\mathcal{S}$ (see \eqref{eq:MoebiusGenerators} and \eqref{eq:swap} for the definitions of $\varrho_\gamma$ and $\mathcal{S}$). Now assume there is $F\in \mathcal{H}(\Omega)$ such that
					\[F\left(\varrho_\gamma(z,w)\right)=F(\gamma z,w/\gamma)=F(z,w)\]
					for all $(z,w)\in\Omega$. Theorem 4.3 in \cite{HeinsMouchaRoth3} implies that $F$ is a scalar multiple of $P_0^{-m}$. Then the claim follows by \eqref{eq:zonaliszeroPoissonForier} and the {$0$-($\Omega$-)homogeneity} of $P_0^{-m}$. For the other option for $T$ apply the previous argument to $F\circ\mathcal{F}$ and use that {$P_0^{-m}\circ\mathcal{F}=P_0^{-m}$} and $\varrho_\gamma\circ\mathcal{F}=\mathcal{F}\circ\varrho_{1/\gamma}$ (see Lemma \ref{lem:MoebiusGenerators} for the definition of $\mathcal{F}$).

					Next choose $(u,v)\in\Omega\cap(\C^*\times\C^*)$. Again by Lemma \ref{lem:TzwProperties}(a) the automorphisms $T\in\mathcal{M}$ that fix the point $(u,v)$ are precisely of the form $T=T_1\circ\varrho_\gamma\circ T_2$ or $T=T_1\circ\varrho_\gamma\circ\mathcal{S}\circ T_2$ where {$T_1,T_2\in\{T_{u,v},\widetilde{T}_{u,v}\}$}. First assume that there is $F\in \mathcal{H}(\Omega)$ such that
					\[F\left((T_{u,v}\circ\varrho_\gamma\circ T_{u,v})(z,w)\right)=F(z,w)\]
					for all $(z,w)\in\Omega$. This condition is equivalent to
					\[F\left((T_{u,v}\circ\varrho_\gamma)(z,w)\right)=F(T_{u,v}(z,w))\]
					for all $(z,w)\in\Omega$. Now we can mimic the argument of the special case and conclude that $(F\circ T_{u,v})(z,w)=\alpha (P_0^{-m}\circ T_{u,v})(z,w)$ for all $(z,w)\in\Omega$ and some $\alpha\in\C$. Thus, the claim follows by \eqref{eq:zonaliszeroPoissonForier}. The other options for $T$ can be treated similarly.

					Note that the cases that either $u$ or $v$ equals $0$ and that either $u$ or $v$ equals $\infty$ are also included in the previous considerations. Moreover, the case $(u,v)=(\infty,\infty)$ is proven analogously using Lemma \ref{lem:TzwProperties}(e).\qedhere
				\end{enumerate}
			\end{proof}
			\begin{remark}
				The automorphisms $\varrho_\gamma$, $\gamma\in\C^*$, precisely correspond to the rotations in $\mathrm{SO}(3,\C)$ that fix the north resp.\ south pole of the sphere. The swap $\mathcal{S}$ translates to reflection at the $(z_2-z_3)$-plane and the flip $\mathcal{F}$ to reflection at the origin.
			\end{remark}
			There are some additional properties of the real zonal harmonics, see e.g.\ \cite[Prop.~5.27]{axler2001} and \cite[Th.~3.5.9]{simon15}, that carry over to the complex case as well.
			\begin{corollary}\label{cor:zonalproperties}
				Let $m\in\N_0$ and $Z_m$ be the complex zonal harmonic of degree $m$. The following properties hold
				\begin{itemize}
					\item[(a)] \emph{Symmetry}: for all $(z,w),(u,v)\in\Omega$ it is $Z_m\big((z,w),(u,v)\big)=Z_m\big((u,v),(z,w)\big)$.
					\item[(b)] \emph{Definiteness}: for all $(z,w)\in\Omega$ it is $Z_m\big((z,w),(z,w)\big)=2m+1$.
					\item[(c)] \emph{Reproducing property}: for all $(z,w)\in\Omega$ and $f\in X_{4m(m+1)}(\Omega)$ we have
					\begin{equation*}
						\frac{i}{\pi}\int\limits_{\widehat{\C}} f(\eta,-\cc{\eta})Z_m\big((z,w),(\eta,-\cc{\eta})\big)\frac{d\eta\,d\cc{\eta}}{(1+|\eta|^2)^2}=f(z,w) \, .
					\end{equation*}
				\end{itemize}
			\end{corollary}
			\begin{proof}
				Part (a) holds by definition. Alternatively, one can use Lemma \ref{lem:scalarproductforautomorphisms}. For Part (b) we use that (whenever the respective quantity is defined)
				\[P_0^{-m}(T_{z,w}(z,w))=P_0^{-m}(\widetilde{T}_{z,w}(z,w))=P_0^{-m}(0,0)=1.\] Part (c) is a consequence of Proposition \ref{prop:Jacobiorthogonality} resp.\ \eqref{eq:Schaudercoefficient} and the identity
				\[P_{-j}^{-m}(u,-\cc{u})=(-1)^j\overline{P_{j}^{-m}(u,-\cc{u})}\,.\qedhere\]
			\end{proof}
			As the last part of this section we investigate how the theory of real zonal harmonics and its hyperbolic counterpart are included in our results. More precisely, we consider the functions in $\mathcal{A}(\widehat{\C})$ and $\mathcal{A}(\D)$ (see \eqref{eq:starproductC} and \eqref{eq:starproductD}) corresponding to our complex zonal harmonics, i.e.\ the restrictions of complex zonal harmonics to the ``rotated diagonal'' $\rotdiagonal$ and the ``diagonal'' $\diagonal$, see \eqref{eq:rotdiagonal} and \eqref{eq:diagonal}. For this purpose we introduce the mappings
			\[\ell_{\widehat{\C}}:\widehat{\C}\to\rotdiagonal,\;\; z\mapsto(z,-\overline{z})\quad\text{and}\quad \ell_\D:\D\to\diagonal,\;\; z\mapsto(z,\overline{z})\,.\]
			Note that the automorphisms in $\mathcal{M}$ that leave $\rotdiagonal$ invariant are precisely of the form $\varrho_\gamma\circ T_{\ell_{\widehat{\C}}(u)}$ or $\varrho_\gamma\circ T_{\ell_{\widehat{\C}}(u)}\circ\mathcal{F}$, for $u\in\C$, $\gamma\in\partial\D$. In other words: the automorphisms induced by \emph{rigid motions of $\widehat{\C}$}. The automorphisms in $\mathcal{M}$ that leave $\diagonal$ invariant are precisely of the form $\varrho_\gamma\circ T_{\ell_\D(u)}$, $u\in\D$, $\gamma\in\partial\D$, i.e.\ the automorphisms induced by \emph{automorphisms of $\D$}. Let $m\in\N_0$. In the following we denote
			\begin{equation*}
				X_{4m(m+1)}(\widehat{\C}):=\{f\in C^\infty(\widehat{\C})\,:\, \Delta_{\widehat{\C}}f=-4m(m+1)f\}\subseteq \mathcal{A}(\widehat{\C})
			\end{equation*}
			and
			\begin{equation*}
				X_{4m(m+1)}(\D):=\{f\in C^\infty(\D)\,:\,\Delta_{\D}f=4m(m+1)f\}\subseteq\mathcal{A}(\D)
			\end{equation*}
			where $\Delta_{\widehat{\C}}$ and $\Delta_\D$ were introduced in \eqref{eq:sphericalLaplaceIntro} and \eqref{eq:hyperbolicLaplaceIntro}.

			\begin{corollary}[Spherical case]\label{cor:zonalsphere}
				Let $m\in\N_0$. The function
				\begin{equation*}
					Z_m^{\widehat{\C}}:\widehat{\C}\times\widehat{\C}\to\C,\quad (z,u)\mapsto Z_m\big(\ell_{\widehat{\C}}(z),\ell_{\widehat{\C}}(u)\big)
				\end{equation*}
				is given by
				\begin{equation}\label{eq:zonalonsphere}
					Z_m^{\widehat{\C}}(z,u)=(2m+1)\left(P_0^{-m}\circ T_{\ell_{\widehat{\C}}(u)}\right)(\ell_{\widehat{\C}}(z))=(2m+1)(-1)^m\Jacobi_m^{(0,0)}\left(\langle \mathbf{x},\mathbf{y}\rangle\right)
				\end{equation}
				where $\mathbf{x}=S(\ell_{\widehat{\C}}(z)),\, \mathbf{y}=S(\ell_{\widehat{\C}}(u))\in\mathbb{S}_\R^2$. Moreover,
				\begin{enumerate}[(a)]
					\item  if either of the variables of $Z_m^{\widehat{\C}}$ is fixed, then the resulting function is in $X_{4m(m+1)}(\widehat{\C})$. Further, if $u\in\widehat{\C}$ is fixed, then $z\mapsto Z_m^{\widehat{\C}}\left(\cdot,u\right)$ is the unique (up to a multiplicative factor) function in $X_{4m(m+1)}(\widehat{\C})$ that is invariant with respect to rigid motions that fix $u$.
					\item \emph{Invariance}: $Z_m^{\widehat{\C}}(\psi(z),\psi(u))=Z_m^{\widehat{\C}}(z,u)$ for every rigid motion $\psi$ of $\widehat{\C}$.
					\item \emph{Symmetry}: $Z_m^{\widehat{\C}}(z,u)=Z_m^{\widehat{\C}}(u,z)$ for all $z,u\in\widehat{\C}$.
					\item \emph{Definiteness}: $Z_m^{\widehat{\C}}(z,z)=2m+1$ for all $z\in\widehat{\C}$.
					\item $Z_m^{\widehat{\C}}$ is real-valued.
					\item let $\{e_j:j=1,\ldots ,2m+1\}$ be a basis of $X_{4m(m+1)}(\widehat{\C})$. If $\{F_j:j=1,\ldots ,2m+1\}$ is a basis of $X_{4m(m+1)}(\Omega)$ orthonormal with respect to $\langle\cdot{,}\cdot\rangle_\Omega$ such that $e_j=F_j\circ \ell_{\widehat{\C}}$ for $j=1,\ldots ,2m+1$, then
					\begin{equation*}
						Z_m^{\widehat{\C}}(z,u)=\sum_{j=1}^{2m+1} \overline{e_j(u)}e_j(z)\,.
					\end{equation*}
				\end{enumerate}
				Additionally,
				\begin{enumerate}[(a)]\setcounter{enumi}{6}
					\item \emph{Reproducing property}: for all $z\in\widehat{\C}$, $f\in X_{4m(m+1)}(\widehat{\C})$ we have
						\begin{equation*}
							\frac{i}{\pi}\int\limits_{\widehat{\C}} f(\eta)Z_m^{\widehat{\C}}(z,\eta)\frac{d\eta\,d\cc{\eta}}{(1+|\eta|^2)^2}=f(z) \, .
					\end{equation*}
					\item for all $z,u\in\widehat{\C}$ it holds that $|Z_m^{\widehat{\C}}(z,u)|\leq 2m+1$.

				\end{enumerate}
			\end{corollary}
			\begin{proof}
				The explicit form \eqref{eq:zonalonsphere} of $Z_m^{\widehat{\C}}$ follows from Proposition \ref{prop:zonaliszeroPoissonForier}, Lemma \ref{lem:Jacobipolynomials} and the considerations in Remark \ref{rem:scalarproductforM} and in Section \ref{sub:complexsphere}. Parts (a) to (d) are direct consequences of the respective properties of $Z_m$, see Theorem \ref{thm:zonalinvariance} and Corollary \ref{cor:zonalproperties}. Part (e) follows from \eqref{eq:zonalonsphere} because $\Jacobi_m^{(0,0)}$ has real coefficients and $\langle \mathbf{x},\mathbf{y}\rangle\in\R$. For Part (f) note that $Z_m^{u,-\cc{u}}:=Z_m(\cdot,(u,-\cc{u}))\in X_{4m(m+1)}(\Omega)$. Thus, we can write $Z_m^{u,-\cc{u}}=\sum_{j=1}^{2m+1}c_j F_j$ where
				\begin{align*}
					c_j&=\frac{i}{\pi}\int\limits_{\widehat{\C}} Z_m^{u,-\cc{u}}(\eta,-\cc{\eta})\cc{F_j(\eta,-\cc{\eta})}\frac{d\eta\,d\cc{\eta}}{(1+|\eta|^2)^2}\\
					&=\frac{i}{\pi}\int\limits_{\widehat{\C}} \cc{F_j(\eta,-\cc{\eta})}Z_m\big((u,-\cc{u}),(\eta,-\eta)\big)\frac{dtds}{(1+s)^2}\\
					&\overset{(e)}{=}\cc{\frac{i}{\pi}\int\limits_{\widehat{\C}} F_j(\eta,-\eta)Z_m\big((u,-\cc{u}),(\eta,-\eta)\big)\frac{dtds}{(1+s)^2}}\overset{\text{Cor. \ref{cor:zonalproperties}(c)}}{=}\cc{F_j(u,-\cc{u})}\, .
				\end{align*}
				It follows
				\[Z_m^{\widehat{\C}}(z,u)=Z_m^{u,-\cc{u}}(\ell_{\widehat{\C}}(z))=\sum_{j=1}^{2m+1}c_j F_j(d_{\widehat{\C}}(z))=\sum_{j=1}^{2m+1}\cc{e_j(u)} e_j(z) \, .\]
				It remains to show the additional claims. The reproducing property follows from Corollary \ref{cor:zonalproperties}. Further, by \eqref{eq:zonalonsphere} Part (h) is equivalent to
				\[|P_0^{-m}(z,-\cc{z})|=\left\vert\Jacobi_{m}^{(0,0)}\left(\frac{1-|z|^2}{1+|z|^2}\right)\right\vert\leq1 \, \]
				for all $z\in\widehat{\C}$. This estimate holds because $x\mapsto\frac{1-x}{1+x}$ maps $[0,\infty]$ to $[-1,1]$ and the Jacobi polynomials can be estimated as follows \cite[Eq.~22.14.1]{abramowitz1984}
				\[|P_n^{(\alpha,\beta)}(x)|\leq\frac{(\alpha+1)_n}{n!} \, ,\qquad \alpha\geq\beta\geq-\frac{1}{2}\, .\qedhere\]
			\end{proof}
			\begin{corollary}[Hyperbolic case]\label{cor:zonalpropertieshyp}
				Let $m\in\N_0$. The function
				\begin{equation*}
					Z_m^{\D}:\D\times\D\to\C,\quad (z,u)\mapsto Z_m\big(\ell_{\D}(z),\ell_{\D}(u)\big)
				\end{equation*}
				is given by
				\begin{equation}
					Z_m^{\D}(z,u)=(2m+1)\left(P_0^{-m}\circ T_{\ell_{\D}(u)}\right)(\ell_{\D}(z))=(2m+1)(-1)^m\Jacobi_m^{(0,0)}\left(\langle \mathbf{x},\mathbf{y}\rangle\right)
				\end{equation}
				where $\mathbf{x}=S(\ell_{\D}(z)),\, \mathbf{y}=S(\ell_{\D}(u))$ are in the lower half of $ H_\R^2$ (see Remark \ref{rem:hyperboloid}). Moreover,
				\begin{enumerate}[(a)]
					\item if either of the variables of $Z_m^{\D}$ is fixed, then the resulting function is in $X_{4m(m+1)}(\D)$. Further, if $u\in\D$ is fixed, then $z\mapsto Z_m^{\D}\left(\cdot,u\right)$ is the unique (up to a multiplicative factor) function in $X_{4m(m+1)}(\D)$ that is invariant with respect to $\psi\in\Aut(\D)$ that fix $u$.
					\item \emph{Symmetry}: $Z_m^{\D}(z,u)=Z_m^{\D}(u,z)$ for all $z,u\in\D$.
					\item \emph{Invariance}: $Z_m^{\D}(\psi(z),\psi(u))=Z_m^{\D}(z,u)$ for every $\psi\in\mathrm{Aut}(\D)$.
					\item \emph{Definiteness}: $Z_m^{\D}(z,z)=2m+1$ for all $z\in\D$.
					\item $Z_m^{\D}$ is real-valued.
					\item let $\{e_j:j=1,\ldots ,2m+1\}$ be a basis of $X_{4m(m+1)}(\D)$. If $\{F_j:j=1,\ldots ,2m+1\}$ is a basis of $ X_{4m(m+1)}(\Omega)$ orthonormal with respect to $\langle\cdot{,}\cdot\rangle_\Omega$ such that $e_j=F_j\circ \ell_{\D}$ for $j=1,\ldots ,2m+1$, then
					\begin{equation*}
						Z_m^{\D}(u,v)=\sum_{j=1}^{2m+1} \overline{e_j(u)}e_j(z)\,.
					\end{equation*}
				\end{enumerate}
			\end{corollary}
			The proof is analogous to the one of Corollary \ref{cor:zonalsphere}. Note that the assumptions on the set ${\mathcal{B}:=\{F_j\,:\,j=1,\ldots,2m+1\}}$ in Part (f) of both Corollary \ref{cor:zonalsphere} and \ref{cor:zonalpropertieshyp} are no big restrictions since $\mathcal{B}$ automatically becomes a basis of $X_{4m(m+1)}(\Omega)$ if the set ${\{e_j\,:\,j=1,\ldots,2m+1\}}$ is a basis of $X_{4m(m+1)}(\widehat{\C})$ resp.\ $X_{4m(m+1)}(\D)$. This is a consequence of the one-to-one correspondence of the eigenspaces of $\Delta_{zw}$ and $\Delta_{\widehat{\C}}$ resp.\ $\Delta_{\D}$ (when interpreting $\Delta_{zw}$ as operator on $\mathcal{H}(\D\times\D)$), see \cite[Th. 2.2 and 2.5]{HeinsMouchaRoth3}.

			\section{A combinatorial identity}\label{sec:coefficients}

			Recall that we postponed the crucial part of the proof of Lemma \ref{lem:OldSchauderasPFM}, namely the verification of the combinatorial identity \eqref{eq:combinatorialID}.
			\begin{theorem}
				\label{th:CombinatorialIDIntro}
				Let $m,n,d\in\N_0$ such that $m\geq n\geq d$. Then the following identity is true	\begin{equation*}\label{eq:coefficientID}
					\sum_{s=0}^n\sum_{t=0}^{d}	(-1)^{s+t}\binom{m}{s}\binom{s}{t}\binom{2m-s}{n}^{-1}\binom{m-s}{d-t+m-n}\binom{m-s}{d-t}\frac{2m-2s+1}{2m-s+1}=\delta_{n,d} \, .\tag{CID}
				\end{equation*}
				Here, $\delta_{n,d}:=1$ if $n=d$ and zero otherwise.
			\end{theorem}
			An essential step for the proof of Theorem \ref{th:CombinatorialIDIntro} is rewriting the inner sum as a ${}_{3}F_{2}$ hypergeometric series.
			\begin{lemma}\label{lem:tSum}
				Let $m,n,d,s\in\N_0$ such that $m\geq n\geq d,s$. Then the following identity is true
				\begin{align*}\label{eq:tsum}
					\sum_{t=0}^{d}&(-1)^{t}\binom{s}{t}\binom{m-s}{d-t+m-n}\binom{m-s}{d-t}\\
					&=(-1)^{d-n+s}\binom{m-s}{n-s}\binom{n}{d}\frac{s!(2m-s+1)!}{n!(2m-n+1)!}\FFz{s-n,2m-n-s+1,m-d+1}{2m-n+2,m-n+1}{1}\, .\tag{L1}
				\end{align*}
			\end{lemma}
			See \eqref{eq:generalFpq} for the definition of ${}_3F_2$. In Section \ref{sec:identities} some well-known identities for generalized hypergeometric and Beta functions required in the proof of \eqref{eq:coefficientID} are collected. After that, Section \ref{sec:proof} contains the proof of Theorem \ref{th:CombinatorialIDIntro}. Although Lemma \ref{lem:tSum} is used in Section \ref{sec:proof}, we prove it only afterwards, because then some quite technical concept will be introduced that is not required to understand Section \ref{sec:proof}. 
			Note that this chapter is self-contained as the methods we use are not directly related to the PFM or $\mathcal{H}(\Omega)$.

			\subsection{The Hypergeometric and the Beta function}\label{sec:identities}\

			\medskip

			For $p,q\in\N_0$, the \emph{generalized hypergeometric function} ${}_{p}F_{q}$ is defined by
			\begin{equation}\label{eq:generalFpq}
				\pFq{p}{q}{a_1,\dots,a_p}{b_1,\dots,b_q}{z}:=\sum_{k=0}^\infty\frac{(a_1)_k\cdots(a_p)_k}{(b_1)_k\cdots(b_q)_k}\frac{z^k}{k!}
			\end{equation}
			for $a_1,\dots,a_p,b_1,\dots,b_q,z\in\C$, $|z|<1$ and none of $b_1,\dots,b_q$ being a negative integer or zero. In many cases it is possible to extend the definition to more values of $z$ and $b_1,\dots,b_q$. In particular, if some $a_j=-n$ for an integer $n\in\N_0$, then ${}_{p}F_{q}$ is a polynomial of degree $n$ and all values of $z\in\C$ as well as $b_j\leq -n$ are permitted. This will be the situation in our considerations, because when proving a combinatorial identity, one way how hypergeometric series could come into play is identifying the finite combinatorial sum as a (terminating) ${}_{p}F_{q}$ series at the point 1. Further, note that ${}_{2}F_{1}$ coincides with \eqref{eq:hyp2F1}.

			Changing the order of the $a_1,\dots, a_p$ does not change the hypergeometric series. The same is true for the $b_1,\dots,b_q$. Further, whenever there is some $a_j$ that coincides with some $b_k$, both parameters cancel out and the ${}_{p}F_{q}$ series is in fact a ${}_{p-1}F_{q-1}$ series, i.e.\
			\begin{equation*}
				\pFq{p+1}{q+1}{a_1,\dots,a_p,c}{b_1,\dots,b_q,c}{z}=\pFq{p}{q}{a_1,\dots,a_p}{b_1,\dots,b_q}{z} \, .
			\end{equation*}
			This phenomenon can be generalized to parameters coinciding up to a positive integer shift, i.e.			\begin{equation}\label{eq:Fpqshift}
				\pFq{p+1}{q+1}{a_1,\dots,a_p,c+n}{b_1,\dots,b_q,c}{z}=\sum\limits_{j=0}^{n}\binom{n}{j}\frac{(a_1)_j\cdots(a_p)_j}{(b_1)_j\cdots(b_q)_j(c)_j}\pFq{p}{q}{a_1+j,\dots,a_p+j}{b_1+j,\dots,b_q+j}{z}
			\end{equation}
			for all $n\in\N_0$, see \cite[p.~439, 15.]{prudnikov90}.
			\begin{remark}
				The cited reference is a table of identities. One way to prove \eqref{eq:Fpqshift} is shown in \cite{fields61}: one can prove their more general formula (1.3) by means of the Laplace transform and deduce their formula (1.6). The latter is \eqref{eq:Fpqshift} for $w=1$, $r=s=1$, $c_1=c+n$ and $d_1=c$.
			\end{remark}
			In view of \eqref{eq:coefficientID} we also have to deal with an inverse binomial coefficient. Here, the \emph{Beta function} $B$ is useful. We define
			\begin{equation}\label{eq:Beta}
				B(a,b)=\int\limits_0^1 y^{a-1}(1-y)^{b-1}\,dy \, , \qquad a,b\in\C\,,\; \Re(a),\Re(b)>0 \, .
			\end{equation}
			The connection to inverse binomial coefficients is the following \cite[6.1.21]{abramowitz1984}
			\begin{equation}\label{eq:BetaBinomial}
				\frac{1}{n+1}\binom{n}{k}^{-1}=B(n-k+1,k+1) \, , \qquad n,k\in\N_0 \, .
			\end{equation}
			Further, as a consequence of change of variables, one obtains the \emph{Legendre duplication formula} \cite[6.1.18]{abramowitz1984}
			\begin{equation}\label{eq:Duplication}
				B(a,b)=2\int\limits_0^1 y^{2a-1}(1-y^2)^{b-1}\,dy \, .
			\end{equation}
			Moreover, the Beta function is closely related to the Gamma function by \cite[6.2.2]{abramowitz1984}
			\begin{equation}\label{eq:BetaAsGamma}
				B(a,b)=\frac{\Gamma(a)\Gamma(b)}{\Gamma(a+b)} \, .
			\end{equation}

			\subsection{Proof of Theorem \ref{th:CombinatorialIDIntro}}\label{sec:proof}\

			\medskip

			We compute
			\begin{align*}
				&\sum_{s=0}^n\sum_{t=0}^{d}	(-1)^{s+t}\binom{m}{s}\binom{s}{t}\binom{2m-s}{n}^{-1}\binom{m-s}{d-t+m-n}\binom{m-s}{d-t}\frac{2m-2s+1}{2m-s+1}\\
				&\overset{\eqref{eq:tsum}}{=}\sum_{s=0}^n	(-1)^{s}\binom{m}{s}\binom{2m-s}{n}^{-1}\frac{2m-2s+1}{2m-s+1}(-1)^{d-n+s}\binom{m-s}{n-s}\binom{n}{d}\frac{s!(2m-s+1)!}{n!(2m-n+1)!}\\
				&\hspace{6.75cm}\times \FFz{s-n,2m-n-s+1,m-d+1}{2m-n+2,m-n+1}{1}\\
				&=(-1)^{d-n}\binom{n}{d}\sum_{s=0}^n	(2m-2s+1)\frac{m!(2m-s-n)!}{(n-s)!(2m-n+1)!(m-n)!}\\
				&\hspace{5cm}\times \FFz{s-n,2m-n-s+1,m-n+1+(n-d)}{2m-n+2,m-n+1}{1}\,.
			\end{align*}
			Using \eqref{eq:Fpqshift} we can reduce the ${}_{3}F_{2}$ function to a ${}_{2}F_{1}$ function, and this leads to
			\begin{align*}
				&\phantom{=}(-1)^{d-n}\binom{n}{d}\sum_{s=0}^n	(2m-2s+1)\frac{m!(2m-s-n)!}{(n-s)!(2m-n+1)!(m-n)!}\\
				&\hspace{1cm}\times\sum_{k=0}^{n-d}\binom{n-d}{k}\frac{(s-n)_k(2m-n-s+1)_k}{(2m-n+2)_k(m-n+1)_k}\pFq{2}{1}{s-n+k,2m-n-s+1+k}{2m-n+2+k}{1}\\
				&=(-1)^{d-n}\binom{n}{d}\sum_{s=0}^n(2m-2s+1)m!\sum_{k=0}^{n-d}\binom{n-d}{k}\frac{(-1)^k(2m-n-s+k)!}{(n-s-k)!(m-n+k)!(2m-n+1+k)!}\\
				&\hspace{7.5cm}\times\pFq{2}{1}{s-n+k,2m-n-s+1+k}{2m-n+2+k}{1}\,.
			\end{align*}
			Now we use the Gauss identity \cite[Eq.~15.1.20]{abramowitz1984} in order to evaluate the $\F$ functions and obtain
			\begin{align*}
				&\phantom{=}(-1)^{d-n}\binom{n}{d}\sum_{s=0}^n	(2m-2s+1)m!\sum_{k=0}^{n-d}\binom{n-d}{k}\frac{(-1)^k(2m-n-s+k)!}{(n-s-k)!(m-n+k)!(2m-n+1+k)!}\\
				&\hspace{8cm}\times\frac{(2m-n+1+k)!(n-k)!}{s!(2m-s+1)!}\\
				&=(-1)^{d-n}\binom{n}{d}\sum_{s=0}^n\sum_{k=0}^{n-d}(-1)^k\binom{n-d}{k}	\frac{2m-2s+1}{2m-s+1}\binom{2m-s}{n-k}^{-1}\binom{n-k}{s}\binom{m}{n-k}\\
				&=(-1)^{d-n}\binom{n}{d}\sum_{k=0}^{n-d}(-1)^k\binom{n-d}{k}\binom{m}{n-k}\sum_{s=0}^{n-k}\binom{n-k}{s}	\frac{2m-2s+1}{2m-s+1}\binom{2m-s}{n-k}^{-1}\,.
			\end{align*}
			The inverse binomial coefficients can be identified with Beta functions, see \eqref{eq:BetaBinomial},
			\begin{align*}
				&\phantom{\overset{\phantom{\eqref{eq:Beta}}}{=}}(-1)^{d-n}\binom{n}{d}\sum_{k=0}^{n-d}(-1)^k\binom{n-d}{k}\binom{m}{n-k}\\
				&\hspace{4.5cm}\times\sum_{s=0}^{n-k}\binom{n-k}{s}(2m-2s+1)B(2m-s-n+k+1,n-k+1)\\
				&\overset{\eqref{eq:Beta}}{=}(-1)^{d-n}\binom{n}{d}\sum_{k=0}^{n-d}(-1)^k\binom{n-d}{k}\binom{m}{n-k}\\
				&\hspace{4.5cm}\times\sum_{s=0}^{n-k}\binom{n-k}{s}(2m-2s+1)	\int\limits_0^1 y^{2m-s-n+k}(1-y)^{n-k}\,dy\\
				&\overset{\phantom{\eqref{eq:Beta}}}{=}(-1)^{d-n}\binom{n}{d}\sum_{k=0}^{n-d}(-1)^k\binom{n-d}{k}\binom{m}{n-k}\\
				&\hspace{2cm}\times\int\limits_0^1y^{2m-n+k}(1-y)^{n-k}\Bigg[(2m+1)\sum_{s=0}^{n-k}\binom{n-k}{s} y^{-s}dy-2\sum_{s=0}^{n-k}\binom{n-k}{s}sy^{-s}\Bigg]\,dy\,.
			\end{align*}
			We now compute the $s$-sums using the Binomial theorem in its standard and differentiated form
			\begin{align*}
				&\phantom{=}(-1)^{d-n}\binom{n}{d}\sum_{k=0}^{n-d}(-1)^k\binom{n-d}{k}\binom{m}{n-k}\\
				&\hspace{1cm}\times\int\limits_0^1y^{2m-n+k}(1-y)^{n-k}\Bigg[(2m+1)y^{k-n}(1+y)^{n-k}-2(n-k)y^{k-n}(1+y)^{n-k-1}\Bigg]\,dy\\
				&=(-1)^{d-n}\binom{n}{d}\sum_{k=0}^{n-d}(-1)^k\binom{n-d}{k}\binom{m}{n-k}\Bigg[(2m+1)\int\limits_0^1y^{2(m-n+k)}(1-y^2)^{n-k}\,dy\\
				&\hspace{1cm} -2(n-k)\int\limits_0^1y^{2(m-n+k)}(1-y^2)^{n-k-1}\,dy+2(n-k)\int\limits_0^1y^{2(m-n+k)+1}(1-y^2)^{n-k-1}\,dy\Bigg]\,.
			\end{align*}
			Using the Legendre duplication formula \eqref{eq:Duplication} we obtain
			\begin{align*}
				&\phantom{=}(-1)^{d-n}\binom{n}{d}\sum_{k=0}^{n-d}(-1)^k\binom{n-d}{k}\binom{m}{n-k}\Bigg[\frac{2m+1}{2}B\left(m-n+k+\frac{1}{2},n-k+1\right)\\
				&\hspace{3.5cm}-(n-k)B\left(m-n+k+\frac{1}{2},n-k\right)+(n-k)B\left(m-n+k+1,n-k\right)\Bigg]\,.
			\end{align*}
			We can rewrite the Beta functions using Gamma functions, see \eqref{eq:BetaAsGamma}, and proceed with the recursion identity for Gamma functions, that is $\Gamma(z+1)=z\Gamma(z)$, $z\in\C$, which leads to
			\begin{align*}
				&\phantom{=}(-1)^{d-n}\binom{n}{d}\sum_{k=0}^{n-d}(-1)^k\binom{n-d}{k}\binom{m}{n-k}\Bigg[\frac{2m+1}{2}\frac{\Gamma\left(m-n+k+\frac{1}{2}\right)\Gamma(n-k+1)}{\Gamma\left(m+\frac{3}{2}\right)}\\
				&\hspace{2.75cm}-(n-k)\frac{\Gamma\left(m-n+k+\frac{1}{2}\right)\Gamma(n-k)}{\Gamma\left(m+\frac{1}{2}\right)}+(n-k)\frac{\Gamma\left(m-n+k+1\right)\Gamma(n-k)}{\Gamma\left(m+1\right)}\Bigg]\\
				&=(-1)^{d-n}\binom{n}{d}\sum_{k=0}^{n-d}(-1)^k\binom{n-d}{k}\binom{m}{n-k}\Bigg[\frac{\Gamma\left(m-n+k+\frac{1}{2}\right)\Gamma(n-k+1)}{\Gamma\left(m+\frac{1}{2}\right)}\\
				&\hspace{3cm}-\frac{\Gamma\left(m-n+k+\frac{1}{2}\right)\Gamma(n-k+1)}{\Gamma\left(m+\frac{1}{2}\right)}+\frac{\Gamma\left(m-n+k+1\right)\Gamma(n-k+1)}{\Gamma\left(m+1\right)}\Bigg]\, .
			\end{align*}
			Finally, by recognizing the fractions as binomial coefficients and using the Binomial theorem once again we obtain
			\begin{align*}
				(-1)^{d-n}\binom{n}{d}\sum_{k=0}^{n-d}(-1)^k\binom{n-d}{k}\binom{m}{n-k}\binom{m}{n-k}^{-1}=(-1)^{d-n}\binom{n}{d}(1-1)^{n-d}=\delta_{n,d} \, . \qquad \square
			\end{align*}

			\subsection{Proof of Lemma \ref{lem:tSum}}\label{sec:tsum}\

			\medskip

			\subsubsection{Step 1: Rewriting \eqref{eq:tsum}}\

			For $n\geq d+s$ we find that the left-hand side of \eqref{eq:tsum} is given by
			\begin{align*}\label{eq:tsumhard}
				\sum_{t=0}^{d}(-1)^{t}\binom{s}{t}\binom{m-s}{d-t+m-n}&\binom{m-s}{d-t}\\
				&=\binom{m-s}{n-d-s}\binom{m-s}{d}\FFz{-s,-d-m+n,-d}{m-s-d+1,n-d-s+1}{1}\, .\tag{T1}
			\end{align*}
			This is neither the desired ${}_{3}F_{2}$ function in Lemma \ref{lem:tSum} nor makes the expression sense for ${n<d+s}$. Thus, for $n\leq d+s$ we first rewrite the $t$-sum as follows
			\begin{align*}
				\sum_{t=0}^{d}(-1)^{t}\binom{s}{t}\binom{m-s}{d-t+m-n}&\binom{m-s}{d-t}\overset{n\leq d+s}{=}\sum_{t=d+s-n}^{d}(-1)^{t}\binom{s}{t}\binom{m-s}{d-t+m-n}\binom{m-s}{d-t}\\
				=&\sum_{t=0}^{n-s}(-1)^{t-n+d+s}\binom{s}{t+d+s-n}\binom{m-s}{m-s-t}\binom{m-s}{n-t-s}\\
				=&(-1)^{d+s-n}\binom{s}{n-d}\binom{m-s}{n-s}\FFz{s-n,d-n,s-m}{d+s-n+1,m-n+1}{1}\, .\tag{T2}\label{eq:tsumeasy}
			\end{align*}
			Therefore, we need to show that both formulas \eqref{eq:tsumhard} and \eqref{eq:tsumeasy} can be transformed into \eqref{eq:tsum}.

			\subsubsection{Step 2: Rewriting \eqref{eq:tsumeasy}}\

			There are a great many well-known identities for ${}_{3}F_{2}$ functions evaluated at $z=1$. The one we need is the following \cite[Cor. 3.3.5]{askey}
			\begin{equation}\label{eq:rewritingT2}
				\FFz{a,b,c}{e,f}{1}=\frac{\Gamma(f)\Gamma(e+f-a-b-c)}{\Gamma(f-a)\Gamma(e+f-b-c)}\FFz{a,e-b,e-c}{e,e+f-b-c}{1}
			\end{equation}
			which implies for $a=s-n$, $b=s-m$, $c=d-n$, $e=m-n+1$, $f=d+s-n+1$ that
			\begin{align*}
				&\FFz{s-n,s-m,d-n}{m-n+1,d+s-n+1}{1}\\
				&\hspace{3cm}=\frac{\Gamma(d+s-n+1)\Gamma(2m-s+2)}{\Gamma(d+1)\Gamma(2m-n+2)}\FFz{s-n,2m-s-n+1,m-d+1}{m-n+1,2m-n+2}{1} \, .
			\end{align*}
			Plugging this identity in \eqref{eq:tsumeasy} proves \eqref{eq:tsum} for $n\leq d+s$.

			\subsubsection{Step 3: Whipple's theory}\

			Rewriting \eqref{eq:tsumhard} requires some deeper results about ${}_{3}F_{2}$ functions which go back to Whipple \cite{whipple}. Note that in the literature, many identities concerning ${}_{3}F_{2}$ functions with unit argument can be found. For instance, see \eqref{eq:rewritingT2}, \cite[Cor.\ 3.3.4, 3.3.6]{askey} or \cite[(2) on p. 15]{bailey}. Given parameters $a,b,c,e,f\in\C$ all those identities contain two ${}_{3}F_{2}$ functions depending on (sums and differences) of these parameters and (usually) six Gamma functions which also depend on the parameters. Whipple used this observation and developed a ``recipe'' how to obtain a large number of transformations of this type. For convenience of the reader, we briefly introduce Whipple's theory following \cite{whipple} and \cite[Sec.\ 3.5-3.9]{bailey} as far as it is needed for proving Lemma \ref{lem:tSum}.

			\paragraph{\textit{The notation}}\

			We want to identify $\FFz{a,b,c}{e,f}{1}$ as a special case of ${}_{3}F_{2}$ functions with other combinations, i.e.\ sums and differences, of the parameters $a,b,c,e,f$. For this define
			\begin{align*}\label{eq:Fp}
				\textit{Fp}(u;v,w)&:=\frac{1}{\Gamma(\alpha_{xyz})\Gamma(\beta_{vu})\Gamma(\beta_{wu})}\FFz{\alpha_{vwx},\alpha_{vwy},\alpha_{vwz}}{\beta_{vu},\beta_{wu}}{1}\tag{Fp}\\
				\textit{Fn}(u;v,w)&:=\frac{1}{\Gamma(\alpha_{uvw})\Gamma(\beta_{uv})\Gamma(\beta_{uw})}\FFz{\alpha_{uyz},\alpha_{uzx},\alpha_{uxy}}{\beta_{uv},\beta_{uw}}{1}\label{eq:Fn}\tag{Fn}
			\end{align*}
			where
			\begin{align*}
				\alpha_{\ell mn}&:=\frac{1}{2}+r_\ell+r_m+r_n\quad\text{and}\quad
				\beta_{mn}:=1+r_m-r_n
			\end{align*}
			for some parameters $r_0,r_1,r_2,r_3,r_4,r_5$ such that $\sum_{k=0}^5 r_k=0$. We always assume the indices to be pairwise distinct, i.e.\ $\{u,v,w,x,y,z\}=\{0,1,2,3,4,5\}$.

			At first glance, this might seem rather complicated. However, it is only important to note three things: first, the $\alpha$-coefficients are invariant under permutation of the indices whereas the $\beta$-coefficients are not. Second, we will not work with the $r$-coefficients, we just need to know that the $r$-coefficients can be chosen to depend on $a,b,c,e,f$ only. In particular, in the following we choose
			\begin{align*}
				\alpha_{145}=a,\quad \alpha_{245}=b,\quad\alpha_{345}=c,\quad\beta_{40}=e,\quad\beta_{50}=f,\quad\alpha_{123}=e+f-a-b-c=s \, .
			\end{align*}
			Third, in this case the $\textit{Fp}$ and $\textit{Fn}$ functions correspond to ${}_{3}F_{2}$ functions depending on sums and differences of $a,b,c,e,f$ and
			\begin{equation*}
				\textit{Fp}(0;4,5)=\frac{1}{\Gamma(s)\Gamma(e)\Gamma(f)}\FFz{a,b,c}{e,f}{1} \, .
			\end{equation*}
			Whipple \cite[Tab. I, IIA, IIB]{whipple} provides tables with all $\alpha$- and $\beta$-coefficients in terms of $a,b,c,e,f$ and some examples of explicit $\textit{Fp}$ and $\textit{Fn}$ functions. The achievement of Whipple was to find relations between $\textit{Fp}$ and $\textit{Fn}$ functions of different parameters.

			\paragraph{\textit{General relations}}\

			The first observation that we need is that, in fact, the parameter $u$ in \eqref{eq:Fp} divides the 60 possible $\textit{Fp}$ functions into 5 types. It turns out that given $u\in\{0,1,2,3,4,5\}$, all ten $\textit{Fp}(u;v,w)$ functions coincide. For example, if $u=0$, this follows from \eqref{eq:rewritingT2}, \cite[(2) on p. 15]{bailey} and using the permutation invariance of ${}_{3}F_{2}$ functions with respect to the parameters $a,b,c$ resp.\ $e,f$.

			The same is true for the $\textit{Fn}$ functions, therefore, we can write $\textit{Fp}(u;v,w):=\textit{Fp}(u)$ and ${\textit{Fn}(u;v,w):=\textit{Fn}(u)}$ for all $v,w\in\{0,1,2,3,4,5\}\setminus\{u\}$. This allows us to easily obtain new identities for ${}_{3}F_{2}$ functions by identifying two ${}_{3}F_{2}$ functions as the same $\textit{Fp}(u)$ resp.\ $\textit{Fn}(u)$ function. For example, in this notion \eqref{eq:rewritingT2} is equivalent to
			\[\textit{Fp}(0;4,5)=\textit{Fp}(0)=\textit{Fp}(0;1,4) \, .\]
			As we will see in Section \ref{sec:rewritingT1}, rewriting \eqref{eq:tsumhard} requires identities between two $\textit{Fp}(u)$ and $\textit{Fp}(\widetilde{u})$ functions for $u\neq\widetilde{u}$. Whipple \cite[Sec. 8]{whipple} provides ``three-term relations'', i.e.\ relations between three $\textit{Fp}$ or $\textit{Fn}$ functions. Since we are only interested in the case of integer parameters $a,b,c,e,f$, these three-term relations simplify to two-term relations \cite[Sec. 5]{whipple}.

			\paragraph{\textit{Integer case}}\

			Whipple (and Bailey) treat the case of only $c$ being a negative integer, whereas in our case all parameters $a,b,c,e,f$ are integers. Raynal \cite[Sec. 2-5]{raynal78} considers this case. In our notion, the identity that we need to finish the proof of Lemma \ref{lem:tSum} is the following (see \cite[(46)]{raynal78})
			\begin{equation}\label{eq:raynal}
				Rp(0)\textit{Fp}(0)=(-1)^{\beta_{05}-1}Rp(5)\textit{Fp}(5)
			\end{equation}
			where
			\begin{equation*}
				R_p(0)^2=\frac{\Gamma(\alpha_{123})\Gamma(\alpha_{124})\Gamma(\alpha_{125})\Gamma(\alpha_{134})\Gamma(\alpha_{135})\Gamma(\alpha_{234})\Gamma(\alpha_{235})}{\Gamma(\alpha_{012})\Gamma(\alpha_{013})\Gamma(\alpha_{023})}
			\end{equation*}
			and
			\begin{equation*}
				R_p(5)^2=\frac{\Gamma(\alpha_{012})\Gamma(\alpha_{013})\Gamma(\alpha_{023})\Gamma(\alpha_{123})\Gamma(\alpha_{124})\Gamma(\alpha_{134})\Gamma(\alpha_{234})}{\Gamma(\alpha_{125})\Gamma(\alpha_{135})\Gamma(\alpha_{235})} \, .
			\end{equation*}
			In general, $Rp(u)^2$ is given by the product of all $\Gamma(\alpha_{xvw})$ for $x,v,w\neq u$ resp.\ $\Gamma(1-\alpha_{xvw})^{-1}$ if $\alpha_{xvw}$ is a negative integer or zero. Note that $1-\alpha_{xvw}=\alpha_{qst}$ for ${q,s,t}\in\{0,1,2,3,4,5\}\setminus\{x,v,w\}$.

			\subsubsection{Step 4: Rewriting \eqref{eq:tsumhard}}\label{sec:rewritingT1}\

			Set $a=-s$, $b=-d-m+n$, $c=-d$, $e=1-d+m-s$ and $f=1-d+n-s$. Using the notation established in the previous section we recognize the hypergeometric function on the right-hand side of \eqref{eq:tsumhard} as $\textit{Fp}(0;4,5)$ and the hypergeometric function on the right-hand side of~\eqref{eq:tsum} as $\textit{Fp}(5;3,4)$. We can relate those functions by
			\begin{equation*}
				\textit{Fp}(0;4,5)=\textit{Fp}(0)\overset{\eqref{eq:raynal}}{=}(-1)^{d-n+s}\frac{Rp(5)}{Rp(0)}\textit{Fp}(5)=(-1)^{\beta_{05}-1}\frac{Rp(5)}{Rp(0)}\textit{Fp}(5;3,4) \, .
			\end{equation*}
			Plugging in the definitions proves \eqref{eq:tsum} for $n\geq d+s$, and thus Lemma \ref{lem:tSum}.

			\section*{Acknowledgements}

			The author thanks Michael Heins and Oliver Roth for countless helpful and inspiring discussions.

		\end{document}